\documentclass[11pt]{amsart}
\usepackage{color,amssymb, amsbsy,amsmath,amsfonts, amssymb, amscd, epsfig, graphics, mathtools, commath, xcolor}
 
\usepackage{amsthm}
\usepackage{enumitem}
\usepackage[colorlinks=true]{hyperref}
\usepackage{multirow}
\usepackage{tikz}
\usepackage{chemfig}

\newtheorem{theorem}{Theorem}[section]
\newtheorem{proposition}[theorem]{Proposition}

\newtheorem{lemma}[theorem]{Lemma}

\newtheorem{remark}[theorem]{Remark}

\theoremstyle{definition}

\numberwithin{equation}{section}

\numberwithin{theorem}{section}

\DeclareMathOperator{\esssup}{ess\, sup}
\numberwithin{equation}{section}

\newcommand{\blue}[1]{#1}
\newcommand{\bra}[1]{\left(#1\right)}
\newcommand{\R}{\mathbb{R}}

\newcommand{\wh}{\widehat}

\allowdisplaybreaks

\newcommand{\vertiii}[1]{\|\hspace*{-.02in}|#1|\hspace*{-.02in}\|}

\begin{document}
\title[Reactions networks possessing boundary equilibria]{Stability analysis of irreversible chemical reaction-diffusion systems with boundary equilibria}
\author[L.~Nguyen]{Thi Lien Nguyen}
\author[B.Q.~Tang]{Bao Quoc Tang}
	\address{Thi Lien Nguyen\hfill\break
		 Faculty of Mathematics and Informatics, Hanoi National University of Education\hfill\break
		136 Xuan Thuy, Cau Giay, Hanoi, Vietnam
		}
	\email{ntlien@hnue.edu.vn} 	\address{Bao Quoc Tang \hfill\break
		Department of Mathematics and Scientific Computing, University of Graz \hfill\break
        Heinrichstrasse 36, 8010 Graz, Austria}
	\email{quoc.tang@uni-graz.at, baotangquoc@gmail.com}

	\begin{abstract}
        Large time dynamics of reaction-diffusion systems modeling some irreversible reaction networks are investigated. Depending on initial masses, these networks possibly possess boundary equilibria, where some of the chemical concentrations are completely used up. In the absence of these equilibria, we show an explicit convergence to equilibrium by a modified entropy method, where it is shown that reactions in a measurable set with positive measure is sufficient to combine with diffusion and to drive the system towards equilibrium. When the boundary equilibria are present, we show that they are unstable (in Lyapunov sense) using some bootstrap instability technique from fluid mechanics, while the nonlinear stability of the positive equilibrium is proved by exploiting a spectral gap of the linearized operator and the uniform-in-time boundedness of solutions. 
	\end{abstract}
	\maketitle
	\tableofcontents
\section{Introduction}

The study of large time dynamics for reaction-diffusion systems modeling chemical reaction networks has been recently  progressed considerably. This research is partially rooted from the chemical reaction network theory, which was initiated in the seventies with the core idea is to investigate dynamical behavior of large chemical reaction systems using only their network structures, regardless of concrete reaction rate constants, see e.g. the monograph \cite{feinberg2019foundations}. One of the center questions is to determine the convergence towards chemical equilibria, especially for complex balanced or weakly reversible networks. The latter gave rise to the well known Global Attractor Conjecture, which states that, for a large class of reaction systems having a certain structure, the concentrations will evolve towards a positive chemical equilibrium \cite{GAC74}. 
This conjecture has eluded a definitive proof for fifty years, though numerous special cases have been proven, see e.g. \cite{Anderson11,Craciun09,Pantea,Siegel04,Sontag}.
One main obstacle is the existence of (possibly infinitely many) boundary equilibria, which are states where some chemicals are completely used up. It is remarked that all of these aforementioned works considered the spatially homogeneous setting, i.e. the evolution of the network is described by a system of differential equations. When the spatial heterogeneity is taken account, one is led to consider reaction-diffusion systems, whose study is particularly difficult since, mathematically speaking, they are posed in an infinite dimensional space. Even the global existence of a suitable solution is in general unsolved, let alone the large time dynamics. Nevertheless, considerable progress has been obtained when there are no boundary equilibria, see e.g. \cite{glitzky1996free,glitzky1997energetic,groger1986existence,groger1983asymptotic,groger1992free} for qualitative results concerning quadratic systems in one or two dimensions; \cite{desvillettes2006exponential,desvillettes2008entropy,fellner2016exponential} for quantitative results for special systems by using entropy method; or \cite{DFT17,FT18,mielke2015uniform} for general systems, where \cite{FT18} showed exponential convergence to equilibrium for renormalized solutions, which, to our knowledge, is the most general result to date for systems without boundary equilibria. 
When the equilibria are present, only special cases have been investigated. This is not surprising since even in the ODE setting, showing convergence to the positive equilibrium in presence of boundary equilibria is challenging. In \cite{craciun2017convergence,DFT17,FT18}, by exploiting an entropy method and showing that the trajectory could only converge ``slowly'' towards the boundary $\partial\mathbb R_+^m$, some special case of the Global Attractor Conjecture was shown in the PDE setting. In the recent result \cite{Jin}, the (Lyapunov) instability of boundary equilibrium for the reaction network $A + 2B \leftrightarrows B + C$ was shown using some bootstrap instability technique. While (weakly) reversible networks have been extensively investigated, large time dynamics of irreversible networks, especially in the PDE setting is much less studied. Up to our knowledge, the only existing result is the recent work \cite{braukhoff2022quantitative}, where the convergence to equilibrium was proved for the irreversible enzyme reaction using some cut-off entropy method. 

\medskip
In this paper, we investigate stability/instability of some irreversible reaction networks which possess positive and/or boundary equilibria, depending on initial masses. More precisely, we consider for three chemicals named $\mathcal A, \mathcal B, \mathcal{C}$, the following {\it irreversible} reaction network
	\begin{equation}\tag{P1}\label{P1}
		\quad \mathcal A+\mathcal B \xlongrightarrow[ ]{k_1} \mathcal C;\quad \mathcal B+ \mathcal C \xlongrightarrow[ ]{k_2} \mathcal A + 2 \mathcal{B}
	\end{equation}
 where the chemical $\mathcal{B}$ can be considered as a catalyst in the chemical reaction between $\mathcal A$ and $\mathcal{C}$. Therefore, the absence of $\mathcal{B}$ implies that no reaction occurs, which partly explains why the reaction-diffusion system generated by the problem \eqref{P1} has a boundary equilibrium. Interestingly, as we will see shortly, whether or not a boundary equilibrium exists depends on the initial masses. To model \eqref{P1}, we assume that the reactions take place in a bounded domain $\Omega\subset\mathbb R^n$  and without loss of generality, by re-scaling  $x\to |\Omega|^{1/n}x,$ we may assume that $|\Omega| =1$, which {\it we will assume throughout this paper}. Denote by $a(x,t),b(x,t),c(x,t)$ the concentrations of substances $\mathcal{ A}, \mathcal{B}, \mathcal{C}$, respectively, at $(x,t)\in\Omega\times\mathbb R_+.$
 Using a suitable rescaling, we can normalize the reaction rate constants, i.e. $k_1 = k_2=1$. By using the mass action law, the corresponding reaction-diffusion system for \eqref{P1} with respect to  homogeneous Neumann boundary condition then reads as
\begin{equation}\label{p1}
   \begin{cases}
 a_t-d_1\Delta a=b(c-a),  &\text { in } \Omega \times\mathbb R_+,\\
 b_t-d_2 \Delta b = b(c-a), &\text { in } \Omega \times\mathbb R_+,\\
 c_t-d_3\Delta c = -b(c-a),  &\text { in } \Omega \times\mathbb R_+,\\
 \nabla a \cdot \nu = \nabla b\cdot \nu = \nabla c\cdot \nu =0, &\text { on } \partial\Omega \times\mathbb R_+\\
 a(x,0)=a_0(x);b(x,0)=b_0(x);c(x,0)=c_0(x),  &\text { in } \Omega,
\end{cases} 
\end{equation}
where $d_1,d_2,d_3$ are positive diffusion coefficients and  $\nu$ is outward normal unit vector on $\partial\Omega$. 
Consider the Wegscheider  matrix (or stoichiometric matrix) of \eqref{P1}
\[
W= \begin{pmatrix} -1 & -1&1 \\1  &1&-1 \end{pmatrix}
\]
 we have $\dim (\ker (W))=2.$ Therefore,  there are exactly two linearly independent conservation laws for solutions of \eqref{p1}, which can be obtained by using the homogeneous Neumann boundary conditions,
\[
\dfrac{d}{d t}   \int_\Omega \left (a(x,t) +c(x,t)\right )dx   =0\quad \text{ and } \quad \dfrac{d}{d t}   \int_\Omega \left (b(x,t) +c(x,t)\right )dx   =0,
\]
and thus
\begin{equation} \label{cl}
\begin{aligned}
      \int_\Omega (a(x,t)+ c(x,t) )dx &= \int_\Omega (a_0(x) +c_0(x) )dx :=M_1, \quad \forall t>0,\\
    \int_\Omega (b(x,t)+ c(x,t) )dx &= \int_\Omega (b_0(x) +c_0(x) )dx :=M_2, \quad \forall t>0,
\end{aligned}
\end{equation}
where  $M_1,M_2>0$ are the initial masses.

A chemical equilibrium for \eqref{p1} for given initial data, is a spatially homogeneous state $(a_\infty,b_\infty,c_\infty)$ that balances the reactions and satisfies the conservation laws \eqref{cl}, i.e.
\begin{equation*}
    \begin{cases}
        a_\infty+c_\infty=M_1,\\
        b_\infty+c_\infty=M_2,\\
        b_\infty(c_\infty-a_\infty)=0.
    \end{cases}
\end{equation*}
Naturally, we are only interested in non-negative equilibria. It is interesting that the existence of the boundary equilibrium and the positive equilibrium depends on the relationship between $M_1$ and $M_2$. Indeed, from the third relation, we have $b_\infty = 0$ or $a_{\infty} = c_\infty$. If $b_\infty=0$ then $c_\infty=M_2.$  Thus, in this case, the non-negative boundary equilibrium $  (\bar a_\infty,\bar b_\infty,\bar c_\infty)= (M_1-M_2,0,M_2)$ exists if and only if $M_1\geq M_2$. If $a_\infty=c_\infty$ then $b_\infty=M_2-\frac{M_1}{2}$ is positive only if $M_1<2M_2.$ Thus, in this case, if $M_1<2M_2, $ the problem \eqref{p1} has a positive equilibrium $  (a^*_\infty,b^*_\infty,c^*_\infty)=(\frac{M_1}{2},M_2-\frac{M_1}{2} , \frac{M_1}{2})$.
The Table below gives the summary of equilibrium points for the problem \eqref{P1}.  
\begin{table}[h!]
\centering
\begin{tabular}{|p{3cm}|p{4.5cm}|p{4.5cm}|p{3 cm}|}
\hline
  & $M_1<M_2$ & $M_2\leq M_1<2M_2$ & $M_1\geq 2M_2$ \\ \hline
Positive equil. &    $(M_1/2,M_2-M_1/2,M_1/2)$ & $(M_1/2,M_2-M_1/2,M_1/2)$ &  non-existence \\ \hline
Boundary equil. & non-existence     &  $ (M_1-M_2,0,M_2)$ & $(M_1-M_2,0,M_2)$ \\ \hline
\end{tabular}
\end{table}
We will deal with these three cases separately.

\medskip
\begin{itemize}
    \item When $M_1 < M_2$, since there exists only one positive equilibrium $(a_\infty^*, b_\infty^*, c_\infty^*) = (M_1/2, M_2 - M_1/2, M_1/2)$, it is natural to expect that this equilibrium is attracting all solutions. Yet it is not straightforward due to the irreversibility of the system, which does not give an apparent dissipative structure. Our idea is to consider a relative entropy functional but \textit{only for $a$ and $c$}, 
    \begin{equation}\label{relative_entropy_ac}
        \mathcal{E}(a,c|a^*_\infty,c^*_\infty)=\int_\Omega \left(a\ln \dfrac{a}{a^*_\infty}-a+a^*_\infty   +     c\ln \dfrac{c}{c^*_\infty}-c+c^*_{\infty}\right)dx
    \end{equation}
    which dissipates along trajectories with the dissipation
    \begin{equation}\label{entropy_dissipation}
\begin{aligned}
     \mathcal{D}(a,c) = -\dfrac{d  }{dt}\mathcal{E}(a,c|a^*_\infty,c^*_\infty) 
      =d_1 \int_\Omega \dfrac{|\nabla a|^2}{a} dx +d_3\int_\Omega \dfrac{|\nabla c|^2}{c} dx + \int_\Omega b(c-a)\ln \dfrac{c}{a}dx\ge 0.
\end{aligned}
\end{equation}
If one would have a uniform positive pointwise lower bound $b(x,t) \ge \underline{b}>0$ for all $(x,t)\in\Omega\times \R_+$, then an entropy-entropy dissipation inequality of the form $\mathcal{D}(a,c) \ge \lambda \mathcal{E}(a,c|a_\infty^*,c_\infty^*)$ follows immediately from e.g. \cite{FT18}. Unfortunately, showing such a lower bound of $b$ is challenging, and it is not clear if such a bound is indeed true. Therefore, we resort to a different strategy. From the conservation laws, we get $\int_{\Omega}b(x,t)dx \ge M_2 - M_1 >0$. This, in combination with the boundedness of $b$, yields that for each $t$, $b(\cdot,t)$ is bounded below in some measurable subset $\omega(t)$ of $\Omega$ with positive measure. Still this set $\omega(t)$ depends on $t$ and might be very rough. To overcome this issue, we observe that a lower bound of $b$ on $\omega(t)$ means both reactions in \eqref{P1} happen in the subset $\omega(t)$. This is enough, using the ideas from the recent work \cite{DDB24}, to obtain an entropy-entropy dissipation inequality, where $\omega(t)$ is required only to be {\it measurable} and the constants in the inequality depends only on the measure of $\omega(t)$ but not its regularity. Thus, we obtain first the explicit, exponential convergence of $a$ and $c$ towards equilibrium, and the convergence of $b$ follows from that and the conservation laws.

\medskip
\item When $M_1 \ge 2M_2$, due to the existence of only boundary equilibrium $(\bar a_\infty, \bar b_\infty, \bar c_\infty) = (M_1 - M_2, 0, M_2)$, we also expect it to be globally attracting. However, obtaining a global, explicit convergence towards this equilibrium is challenging. We are able, nevertheless, to show first a qualitative convergence $(a(t), b(t), c(t)) \to (\bar a_\infty, \bar b_\infty, \bar c_\infty)$ as $t\to\infty$ using the dissipation of a suitable energy function. Then, when the solution is close enough to $(\bar a_\infty, \bar b_\infty, \bar c_\infty)$, we show that the convergence becomes exponential by showing a spectral gap for the linear part, which in turn dominates the nonlinear part in the large time.

\medskip
\item When $M_2 \le M_1 < 2M_2$, there are both positive and boundary equilibria for \eqref{p1}. We are then in a situation similar to the Global Attractor Conjecture. For the positive equilibrium, we can use again a linearization to show its local, exponential stability thanks to a spectral gap and boundedness of solutions. To deal with boundary equilibrium, showing some ``slow convergence towards boundary'' as in \cite{DFT17,FT18} seems not possible. Instead, we will show that this boundary equilibrium is {\it unstable in the sense of Lyapunov}. 
In order to do that, we exploit the bootstrap instability technique, that is initiated for problems in fluid mechanics \cite{Gou1} and has been applied to chemotaxis systems  in \cite{JMB17},\cite{JDE23},\cite{SIAM23}, and also to reversible chemical reaction-diffusion systems with boundary equlibrium in \cite{Jin}. The main idea of this method is to use the precise, fast growth of the linear part to show that, on the time scale of the instability, a stronger (Sobolev) norm of the solution is dominated by its weaker norm \cite{Gou1}.

It should be emphasized that, this Lyapunov instability is not enough to conclude the global convergence towards to positive equilibrium, since we cannot yet rule out possible convergence towards the boundary. Nevertheless, the bootstrap instability is sufficiently robust and can be applied to other problems to obtain the instability of boundary equilibria, as shown in the following parts. These results certainly bring more insights for dynamics of chemical reaction networks around boundary equilibria, especially in the PDE setting.
\end{itemize}

\medskip
To demonstrate the robustness of the bootstrap instability technique, we consider another reaction network consisting of the following reactions
\begin{equation}\label{P2}\tag{P2}
 \mathcal A+\mathcal B \xlongrightarrow[ ]{k_3} 2 \mathcal C;\  \mathcal B+ \mathcal C \xlongrightarrow[ ]{k_4} 2 \mathcal A;\  \mathcal C+\mathcal A \xlongrightarrow[ ]{k_5} 2 \mathcal B.
 \end{equation}
 which yields the following reaction-diffusion system, by assuming again for simplicity $k_3 = k_4 = k_5 = 1$ \begin{equation}\label{eq1'}
		\begin{cases}
			\partial_t a-d_1\Delta a=-ab-ac+2bc,\quad &\text { in } \Omega \times\mathbb R_+,\\
			\partial_t b-d_2 \Delta b = -ba-bc+2ac,\quad &\text { in } \Omega \times\mathbb R_+,\\
			\partial_t c-d_3\Delta c = -ca-cb+2ab, \quad &\text { in } \Omega \times\mathbb R_+,\\
   \nabla a\cdot \nu = \nabla b\cdot \nu = \nabla c\cdot \nu =0, &\text { on } \partial\Omega \times\mathbb R_+\\
 a(x,0)=a_0(x);b(x,0)=b_0(x);c(x,0)=c_0(x),  &\text { in } \Omega,
		\end{cases} 
	\end{equation}
Unlike \eqref{P1}, the reaction network \eqref{P2} possess certain symmetry, and therefore still features a decreasing Boltzmann entropy (see Remark \ref{remark:Boltzmann_entropy}). Yet it is not sufficient to get global convergence to positive equilibrium due to the existence of boundary equilibria, i.e. we are in the same situation as \eqref{P1} for the case $M_2 \le M_1 < 2M_2$. By summing up the equations in \eqref{eq1'} and using the homogeneous Neumann boundary condition, it yields the conservation of total mass
\begin{equation*} 
    \int_\Omega (a(x,t)+b(x,t)+c(x,t) )dx = \int_\Omega (a_0(x)+b_0(x)+c_0(x) )dx :=M,
\end{equation*}
where  $M>0$ is the initial total mass. It can be checked with Wegscheider matrix that it is the only conservation of mass for this system. By looking at the equilibrium system
\begin{equation*}
    \begin{cases}
        -a_\infty b_\infty - a_\infty c_\infty + 2b_\infty c_\infty =0,\\
        -b_\infty a_\infty - b_\infty c_\infty + 2a_\infty c_\infty =0,\\
        a_\infty + b_\infty + c_\infty = M,
    \end{cases}
\end{equation*}
we see that, for any given positive initial total mass $M>0$, \eqref{eq1'} possesses a unique positive equlibrium  $(a^*_\infty,b^*_\infty,c^*_\infty)=(M/3,M/3,M/3)$ and three non-negative  boundary equilibria $ (\bar a_\infty,\bar b_\infty,\bar c_\infty) \in \{(0,0,M); (M,0,0); (0,M,0)\}$. We will show the local exponential stability of the positive equilbirium by linearization, while the Lyapunov instability of the boundary equilibria will be obtained using the bootstrap instability techniques.

\medskip
{\bf The organization of  this paper.} In the next section, the system \eqref{p1} investigated, where the three cases of equilibria depending on initial mass are considered in subsections \ref{subsec:entropy}, \ref{subsection:qualitative} and \ref{subsection:bootstrap} respectively. Section \ref{sec:symmetric} is devoted to the stability analysis of system \eqref{eq1'}.

\medskip
{\bf Notations.}
 \begin{itemize}
     \item For $1\leq p< +\infty,$ denote by $L^p(\Omega)$ the usual Lebesgue spaces with norm 
     \[
     \norm{u}_p=\left (\int_\Omega |u(x)|^pd x \right )^{1/p}.
     \]
     For $p=\infty$, the norm of $L^\infty(\Omega)$ is denoted by
     \[
      \norm{u}_\infty= \esssup_{x\in \Omega} |u (x)|.
     \]
     \item For   $k\in \mathbb N,$ denote by $H^k(\Omega)$ the classical Sobolev space consisting of  measurable functions $u$  for which its weak derivatives  $D^ju$ up to order $k$ have a finite $L^2(\Omega)$ norm. The norm in $H^k(\Omega)$ as follows
     \[\norm{u}_{H^k(\Omega)} = \left (\sum_{j=0}^k \norm{D^j u }_2^2\right )^{1/2}.\]
     \item For $u: \Omega \to \mathbb R,$ denote by 
     \[
     \bar u =\int_\Omega u(x)dx
     \]
     the spatial average of $u$ (noting that $|\Omega|=1$).
     \item We use the same notation $C$ for different constants throughout this paper.
 \end{itemize}
 
\section{Catalytic-type reactions}
We start with the global existence and boundedness of solutions to \eqref{p1}. We will assume the following 
\begin{enumerate}[label=(A\theenumi),ref=A\theenumi]
	\item \label{A1} $\Omega \subset \mathbb R^n$ is a bounded domain and has smooth boundary $\partial\Omega$ (e.g. belong to $C^{2+\upsilon}$ for some $\upsilon$ positive);
	\item \label{A2} $a_{0},b_0,c_0\in   L^\infty(\Omega)$ and $a_0,b_0,c_0\geq 0$ a.e. in $\Omega$.
\end{enumerate}
\begin{theorem}\label{thm:global-existence}
    Assume \eqref{A1}--\eqref{A2}. Then there exists a unique global classical solution in the following sense: for any $T>0$,
    \begin{equation*}
        a,b,c\in C([0,T),L^p(\Omega) )\cap C^{1,2}((0,T)\times \overline{\Omega}), \quad \forall 1\le p < +\infty,
    \end{equation*}
  satisfying each equation in \eqref{p1} pointwise. Moreover, this solution is bounded uniformly in time, i.e. there is a positive constant $\mathcal{K}$ such that
  \begin{equation}\label{uit-boundedness}
 \norm{a(t)}_{ L^\infty(\Omega)} , \ \norm{b(t)}_{ L^\infty(\Omega)},\ \norm{c(t)}_{ L^\infty(\Omega)}\leq \mathcal K,  \text{ for all } t\ge 0.
 \end{equation}
\end{theorem}
\begin{proof}
    Since the nonlinearities $f_z(\cdot), z\in \{a,b,c\}$ of \eqref{p1} are quadratic, quasi-positive, and satisfies a mass conservation condition, the global existence and uniform-in-time boundedness of a unique classical solution follows from \cite{FMT21}.
\end{proof}

\subsection{Stability of positive equilibrium - Entropy method}\label{subsec:entropy}
In this section, we consider $M_1 < 2M_2,$ then system \eqref{p1} has a   positive equilibrium $ (a^*_\infty,b^*_\infty,c^*_\infty)=(\frac{M_1}{2},M_2-\frac{M_1}{2},\frac{M_1}{2})$.  We consider the relative entropy functional for $a$ and $c$ in \eqref{p1} as
\[
\mathcal{E}(a,c|a^*_\infty,c^*_\infty)=\int_\Omega \left(a\ln \dfrac{a}{a^*_\infty}-a+a^*_\infty   +     c\ln \dfrac{c}{c^*_\infty}-c+c^*_{\infty}\right)dx.
\]
Then, along the trajectory of \eqref{p1}, the entropy dissipation functional is
\begin{equation}\label{t3'}
\begin{aligned}
     \mathcal{D}(a,c) &= -\dfrac{d  }{dt}\mathcal{E}(a,c|a^*_\infty,c^*_\infty) \\
     & =d_1 \int_\Omega \dfrac{|\nabla a|^2}{a} dx +d_3\int_\Omega \dfrac{|\nabla c|^2}{c} dx + \int_\Omega b(c-a)\ln \dfrac{c}{a}dx\\
\end{aligned}
\end{equation}
 Noticing that if we can bound $b$ from below even in a subset $\omega\subset \Omega$ with positive measure, i.e., we only use the   information of the reaction on a measurable subset $\omega$, we can obtain the   exponential convergence to the positive equilibrium. When $M_1<M_2,$ using the mass conservation law, we can show this property for $b$. We need the following important lemma.
 \begin{lemma}\label{lem:degenerate_reaction}
    Let $a, b, c:\Omega\to \R_+$ be functions such that $\|a\|_{\infty}, \|b\|_{\infty}, \|c\|_{\infty}\le \mathcal K$ and $a,b,c$ satisfy the conservation laws \eqref{cl}. Fix a constant $\delta \in (0,1)$ and assume that $\omega\subset \Omega$ is a measurable set with $|\omega| \ge \delta$. Then there exists a constant $\lambda$ depending only on $\mathcal K$, $\delta$, $a_\infty^*$, $b_\infty^*$, $c_\infty^*$, and $\Omega$ such that
    \begin{equation}\label{eq:b1}
        \int_{\Omega}\frac{|\nabla a|^2}{a}dx + \int_{\Omega}\frac{|\nabla c|^2}{c}dx + \int_{\omega}(c-a)^2dx \ge \lambda \mathcal{E}(a, c|a_\infty^*, c_\infty^*).
    \end{equation}
 \end{lemma}
 \begin{remark}\hfill
 \begin{itemize}
     \item It is remarked that we do not assume any regularity on the subdomain $\omega$ except for its measurability. This is important in the application of this lemma in the next part, since we will not know any regularity of $\omega$. 
     \item Following the proof of this lemma, we can see that
     \begin{equation*}
         \frac{1}{\lambda} = C\bra{1 + \frac{1}{\delta}},
     \end{equation*}
     for a constant $C$ independent of $\omega$, which means that the inequality \eqref{eq:b1} breaks down reactions do not happen anywhere, i.e. $|\omega| = 0$.
 \end{itemize}
 \end{remark}
    \begin{proof}
        By applying elementary inequality $x\ln \dfrac{x}{y} -x+y \leq \dfrac{1}{y}(x-y)^2$ for all $x,y>0$, we see that 
        \begin{equation}\label{t4}
           \mathcal{E}(a,c|a^*_\infty,c^*_\infty) \leq \dfrac{1}{a^*_\infty} \int_\Omega \left( (a-a^*_\infty)^2+(c-c^*_\infty)^2   \right)dx. 
        \end{equation}
        Using  the conservation law $2a^*_\infty =2c^*_\infty=\bar a +\Bar c$ and  triangle inequality, we see that  
        \[
        (a-a^*_\infty)^2 =\dfrac{1}{4}(2a -2a^*_\infty)^2 =\dfrac{1}{4}(2a-\bar a -\bar c)^2 \leq \dfrac{1}{4} \left [(a-\bar a)^2 + (a-c)^2 +(c-\bar c)^2 \right],
        \]
        which implies that
        \begin{equation}\label{t5}
            \int_\omega (a-a^*_\infty)^2 dx \leq \dfrac{1}{4} \left ( \norm{a-\bar a}^2_{2} +  \norm{c-\bar c}^2_{2} +\int_\omega (a-c)^2dx\right).
        \end{equation}
        Using the triangle inequality, we estimate
        \begin{equation}\label{t6}
        \begin{aligned}
            \int_\Omega (a-a^*_\infty)^2 dx &=  \norm{a-\bar a}^2_{2} + | \bar a -a^*_\infty|^2 = \norm{a-\bar a}^2_{2} +\dfrac{1}{|\omega| }\int_\omega (\bar a -a^*_\infty)^2 dx \\
            &\leq   \norm{a-\bar a}^2_{2} +\dfrac{2}{|\omega| }\int_\omega (a-\bar a  )^2 dx + \dfrac{2}{ |\omega|}\int_\omega (a -a^*_\infty)^2 dx \\
            & \leq  \left(1+ \dfrac{2}{|\omega|}\right) \norm{a-\bar a}^2_{2} + \dfrac{2}{|\omega|}\int_\omega (a -a^*_\infty)^2 dx.
        \end{aligned}
        \end{equation}
        Combining \eqref{t5} and \eqref{t6}, we derive
        \begin{equation}\label{t7}
            \int_\Omega (a-a^*_\infty)^2 dx \leq \bra{1 + \frac{5}{2|\omega|}}\|a - \bar{a}\|_2^2 +\dfrac{1}{2|\omega|} \norm{c-\bar c}^2_{2} + \dfrac{1}{2|\omega|}\int_\omega (c-a)^2 dx  .
        \end{equation}
        Likewise,
        \begin{equation}\label{t8}
            \int_\Omega (c-c^*_\infty)^2 dx \leq \dfrac{1}{2|\omega|}   \norm{a-\bar a}^2_{2}+ \bra{1 + \frac{5}{2|\omega|}}\norm{c-\bar c}^2_{2} +\dfrac{1}{2|\omega|} \int_\omega (c-a)^2 dx  .
        \end{equation}
        It follows then from \eqref{t4}, \eqref{t7} and \eqref{t8}
        \begin{equation*}
        \begin{aligned}
            \mathcal{E}(a, c|a_{\infty}^*, c_{\infty}^*) &\le \frac{1}{a_\infty^*}\bra{\bra{1+\frac{3}{|\omega|}}\|a - \bar a\|_{2}^2 + \bra{1 + \frac{3}{|\omega|}}\norm{c - \bar c}_2^2 + \frac{1}{|\omega|}\int_{\omega}(c-a)^2dx}\\
            &\le \frac{1}{a_\infty^*}\bra{1 + \frac{3}{|\omega|}}\bra{\norm{a-\bar a}_2^2 + \norm{c - \bar c}_2^2 + \int_{\omega}(c-a)^2dx}.
        \end{aligned}
        \end{equation*}
        Next, thanks to  the Poincar\'e-Wirtinger inequality and the uniform boundedness of $a$, we can estimate
        \begin{equation}\label{t9}
            \norm{a-\bar a}^2_{2} \leq P(\Omega)\int_\Omega | \nabla  a|^2dx \leq \mathcal{K} P(\Omega) \int_\Omega \dfrac{|\nabla a|^2}{a} dx,
        \end{equation}
        where $P(\Omega)$ is the constant in the Poincar\'e-Wirtinger inequality.
        Similarly,
        \begin{equation}\label{t10}
            \norm{c-\bar c}^2_{2}   \leq \mathcal{K} P(\Omega)  \int_\Omega \dfrac{|\nabla c|^2}{c} dx.
        \end{equation}
        Therefore, we finally have
        \begin{equation*}
            \begin{aligned}
                &\int_{\Omega}\frac{|\nabla a|^2}{a}dx + \int_{\Omega}\frac{|\nabla c|^2}{c}dx + \int_{\omega}(c-a)^2dx\\
                &\ge \min\{1; (\mathcal{K}P(\Omega))^{-1}\}\bra{\norm{a-\bar a}_2^2 + \norm{c - \bar c}_2^2 + \int_{\omega}(c-a)^2dx}\\
                &\ge \min\{1; (\mathcal{K}P(\Omega))^{-1}\}a_\infty^*\bra{1+\frac{3}{|\omega|}}^{-1}\mathcal{E}(a,c|a_{\infty}^*,c_{\infty}^*)\\
                &\ge \lambda \mathcal{E}(a,c|a_{\infty}^*,c_{\infty}^*)
            \end{aligned}
        \end{equation*}
        with $$\frac{1}{\lambda} = \frac{1}{\min\{1; (\mathcal{K}P(\Omega))\}a_\infty^*}\bra{1+\frac{3}{\delta}},$$
        which is the desired estimate.
    \end{proof}
\begin{theorem}\label{thm2}
		Assume that $M_1<M_2$ and     \eqref{A1}-\eqref{A2} hold, which implies that the problem \eqref{p1} only has the positive equilibrium $(a^*_\infty,b^*_\infty,c^*_\infty).$ Then the solution to \eqref{p1} globally exponentially converges to this equilibrium. That means, there exist positive constants $C,\alpha_1$ such that 
  \begin{equation}\label{t1}
      \norm{a(t)-a^*_\infty}^2_2+\norm{b(t)-b^*_\infty}^2_2+\norm{c(t)-c^*_\infty}^2_2 \leq C e^{-\alpha_1 t}, \text{ for all }t\geq 0.
  \end{equation}
	\end{theorem}
\begin{proof}
Since $M_1<M_2,$ we infer from the mass conservation laws \eqref{cl} that \[
\int_\Omega b(x,t) dx= \int_\Omega a(x,t) dx + M_2 -M_1 \geq M_2 - M_1 >0. 
\]
Fix $t>0$. Denote by $\omega:= \{x\in \Omega: b(x,t) \ge (M_2-M_1)/2\} \subset \Omega$. We have 
\begin{align*}    
    M_2 - M_1 \le \int_{\Omega}b(x,t)dx &= \int_{\omega}b(x,t)dx + \int_{\Omega\backslash \omega}b(x,t)dx\\
    &\le \mathcal{K}|\omega| + \frac{M_2-M_1}{2}|\Omega\backslash\omega| \le \mathcal{K}|\omega| + \frac{M_2-M_1}{2},
\end{align*}
and therefore
\begin{equation}\label{lower_bound}
    |\omega| \ge \frac{M_2 - M_1}{2\mathcal{K}}.
\end{equation}
We can estimate the entropy functional $\mathcal{D}(a,c)$  from \eqref{t3'} that
\begin{equation}\label{t3}
\begin{aligned}
     \mathcal{D}(a,c) &= -\dfrac{d  }{dt}\mathcal{E}(a,c|a^*_\infty,c^*_\infty) \\
     & =d_1 \int_\Omega \dfrac{|\nabla a|^2}{a} dx +d_3\int_\Omega \dfrac{|\nabla c|^2}{c} dx + \int_\Omega b(c-a)\ln \dfrac{c}{a}dx\\
     &\geq d_1 \int_\Omega \dfrac{|\nabla a|^2}{a} dx +d_3\int_\Omega \dfrac{|\nabla c|^2}{c} dx + \frac{M_2-M_1}{2} \int_\omega  (c-a)\ln \dfrac{c}{a}dx\\
     &\geq \min\left\{ d_1, d_2, \frac{M_2-M_1}{2\mathcal{K}}\right\}\bra{ \int_\Omega \dfrac{|\nabla a|^2}{a} dx +\int_\Omega \dfrac{|\nabla c|^2}{c} dx +\int_\omega  (c-a)^2dx},
\end{aligned}
\end{equation}
where, in the last estimate, we used the  Mean Value Theorem $ \ln \dfrac{c}{a} = \ln c  - \ln a=\dfrac{c-a}{\xi},$ where $\xi$ is between $a$ and $c,$ then,  $|\xi|\leq \mathcal K$ due to the uniform boundedness  of $a$ and $c.$ Thanks to \eqref{lower_bound}, we can apply Lemma \ref{lem:degenerate_reaction} to have
\begin{equation*}
    \mathcal{D}(a, c) \ge \min\left\{ d_1, d_2, \frac{M_2-M_1}{2\mathcal{K}}\right\}\lambda \mathcal{E}(a, c|a_\infty^*, c_\infty^*) =: \kappa_1 \mathcal{E}(a, c|a_\infty^*, c_\infty^*)
\end{equation*}
where {\it the constant $\lambda$ depends only on $\mathcal K, M_2 - M_1, a_\infty^*, b_\infty^*, c_\infty^*$ and $\Omega$}. 
By Gronwall's inequality, one gets
\begin{equation*} 
     \mathcal{E}(a,c|a^*_\infty,c^*_\infty) \leq e^{-\kappa_1 t} \mathcal{E}(a_0,c_0|a^*_\infty,c^*_\infty),\quad \text{for all } t\geq 0.
\end{equation*}
Then,  the elementary inequality \[
x\ln\dfrac{x}{y} -x +y \geq (\sqrt x -\sqrt y)^2,\quad \text{for all }x,y>0,
\]
and the uniformly boundedness of $a,c$ imply that
\begin{equation}\label{t2}
\begin{aligned}
    \norm{a(t)-a^*_\infty}_2^2+\norm{c(t)-c^*_\infty}_2^2 &\leq 4\mathcal{K} \int_\Omega [(\sqrt{a(t)}-\sqrt{a^*_\infty} )^2+(\sqrt{c(t)}-\sqrt{c^*_\infty} )^2]dx\\
    &\leq 4\mathcal{K}  \mathcal{E}(a,c|a^*_\infty,c^*_\infty) \leq C_1 e^{-\kappa_1 t} ,
    \end{aligned}
\end{equation}
where $C_1=4\mathcal{K} \mathcal{E}(a_0,c_0|a^*_\infty,c^*_\infty)$ is a positive constant.

\medskip
In order to estimate $ \norm{b(t)-b^*_\infty}_2^2,$ let us denote by  a new variable $w=(a-a^*_\infty,b-b^*_\infty, c-c^*_\infty)$ noting that $a^*_\infty=c^*_\infty,$ the system for $w$ can be written  as 
	\begin{equation}\label{ep2}
		\begin{cases}
			\partial_t w_1 -d_1\Delta w_1=(w_2+b^*_\infty)(w_3-w_1),&\text { in } \Omega \times\mathbb R_+,\\
			\partial_t w_2 -d_2\Delta w_2=(w_2+b^*_\infty)(w_3-w_1),&\text { in } \Omega \times\mathbb R_+,\\
			\partial_t w_3 -d_3\Delta w_3=-(w_2+b^*_\infty)(w_3-w_1),&\text { in } \Omega \times\mathbb R_+,\\
			\nabla w_i \cdot \nu =0,\ 1\leq i\leq 3, &\text { on } \partial\Omega \times\mathbb R_+,\\
			w(x,0)=(a_0(x)-a^*_\infty,b_0(x)-b^*_\infty,c_0(x)-c^*_\infty),&\text { in } \Omega.
	\end{cases}\end{equation}
 The conservation laws for $w$ read as
 \begin{equation}\label{cl1}
     \int_\Omega (w_1+w_3)dx=\int_\Omega (w_2+w_3)dx=0.
 \end{equation}
Multiplying both sides of the second equation in \eqref{ep2} by $w_2,$ then integrating over $\Omega,$ we have
 \begin{equation*}
     \begin{aligned}
      \dfrac{1}{2}   \dfrac{d}{dt} \norm{w_2}^2_2 +d_2\norm{\nabla w_2}^2_2 &=\int_\Omega w_2(w_2+b^*_\infty)(w_3-w_1)dx\\
      &\leq \mathcal{K} \int_\Omega  |w_2||w_3-w_1|dx\\
      & \leq \mathcal{K} \norm{w_2}_2 \norm{w_3-w_1}_2\\
      &\leq \mathcal{K} (\rho \norm{w_2}_2^2 +\dfrac{1}{4\rho}\norm{w_3-w_1}_2^2 )\\
      &\leq \mathcal{K} (\rho \norm{w_2}_2^2 +\dfrac{C_1}{4\rho} e^{-\kappa_1 t} ) \quad (\text{by \eqref{t2}},
     \end{aligned}
 \end{equation*}
where $\rho>0$ is a small constant which will be chosen later. Furthermore, the Poincar\'e-Wirtinger inequality and the conservation law \eqref{cl1} give
\[
\norm{\nabla w_2}_2^2 \geq P(\Omega) \int_\Omega |w_2-\bar{w}_2|^2dx = P(\Omega)\left( \norm{w_2}_2^2 - \bar{w}_2^2  \right) =  P(\Omega)\left( \norm{w_2}_2^2 - \bar{w}_3^2  \right).
\]
 Consequently,
\begin{equation}\label{t111}
  \dfrac{d}{dt}\norm{w_2}_2^2 \leq  - (2d_2P(\Omega)-2\mathcal{K} \rho )\norm{w_2}_2^2 + 2d_2P(\Omega) \bar{w}_3^2 +\dfrac{\mathcal{K}C_1}{2\rho}   e^{-\kappa_1 t}
\end{equation}
By Holder's inequality and \eqref{t2}, one gets
\[
\bar{w}_3^2 =\left(\int_\Omega w_3 dx\right)^2\leq \int_\Omega w_3^2 dx =\norm{c-c_\infty^*}_2^2 \leq C_1e^{-\kappa_1 t}.
\]
Now choosing $\rho=\dfrac{d_2P(\Omega)}{2\mathcal{K}}$, it follows from \eqref{t111} that
\begin{equation}\label{t0}
  \dfrac{d}{dt}\norm{w_2}_2^2 \leq  -\kappa_2 \norm{w_2}_2^2   +C_2    e^{-\kappa_1 t} ,
\end{equation}
where $ \kappa_2 =\min\{\frac{\kappa_1}{2}, d_2 P(\Omega)\} $ and $C_2= \dfrac{C_1\mathcal{K}^2}{d_2P(\Omega)}+2C_1d_2P(\Omega)$. From this
\begin{equation}\label{t12}
\begin{aligned}
\norm{w_2(t)}_2^2 & \leq  \norm{w_2(\cdot,0)}_2^2 e^{-\kappa_2 t} + \dfrac{C_2}{\kappa_1-\kappa_2}(e^{-\kappa_2 t}- e^{-\kappa_1 t})
\leq C_3 e^{-\kappa_2 t},
\end{aligned}
\end{equation}
where $C_3=\norm{w_2(\cdot,0)}_2^2+ \dfrac{C_2}{\kappa_1-\kappa_2}. $

Combining \eqref{t2} and \eqref{t12}, we verify \eqref{t1} with $\alpha_1 =\min \{\kappa_2,\kappa_1 \}=\kappa_2$ and $C=C_1+C_3.$
The proof is completed.
\end{proof}

\subsection{Stability of boundary equilibrium - Linearization}\label{subsection:qualitative}
This section is devoted to investigate the stability of the boundary equilibrium $(\bar a_\infty,0,\bar c_\infty)=(M_1-M_2,0,M_2)$ in the case $M_1 \geq 2M_2.$   Denoting new variable $u(t)=(a(t)-\bar a_\infty,b(t),c(t)-\bar c_\infty) $ and   $\lambda =\bar a_\infty - \bar c_\infty \geq 0,$
 then the system for $u$ can be written  as 
	\begin{equation}\label{s1}
		\begin{cases}
			\partial_t u_1 -d_1\Delta u_1=u_2(u_3-u_1-\lambda),&x\in \Omega,\quad t>0,\\
			\partial_t u_2 -d_2\Delta u_2=u_2(u_3-u_1-\lambda),&x\in \Omega,\quad t>0,\\
			\partial_t u_3 -d_3\Delta u_3=-u_2(u_3-u_1-\lambda),&x\in \Omega,\quad t>0,\\
			\nabla u_i \cdot \nu =0,\ 1\leq i\leq 3, &x\in\partial \Omega,\ t>0,\\
			u_0(x)=(a_0(x)-\bar a_\infty,b_0(x),c_0(x)-\bar c_\infty),& x\in\Omega.
	\end{cases}\end{equation}
 And the conservation laws for $u$ read as
 \begin{equation}\label{s2}
     \int_\Omega (u_1+u_3)dx=\int_\Omega (u_2+u_3)dx=0.
 \end{equation}
    We first prove the global, qualitative stability of the boundary equilibrium.
  \begin{proposition}\label{pro:qualitative_stability}      
    Assume $M_1 \ge 2M_2$. Then 
    \begin{equation*}
        \lim_{t\to\infty}\bra{\|a(t) - \bar a_\infty\|_{2} + \|b(t)\|_2 + \|c(t) - \bar c_\infty\|_{2}} = 0.
    \end{equation*}
  \end{proposition}
  \begin{proof}
      Multiplying the equations of $a$ and $c$ with $a$ and $c$ in $L^2(\Omega)$, respectively, summing the resultants, then integrating on $(0,t)$ give
      \begin{equation*}
      \begin{aligned}
          \|a(t)\|_2^2 + \|c(t)\|_2^2 &+ 2d_1\int_0^t\|\nabla a(s)\|_{2}^2ds + 2d_3\int_0^t\|\nabla c(s)\|_{2}^2ds\\
          &+ 2\int_0^t\int_{\Omega}b(s)(c(s)-a(s))^2dxds = \|a_0\|_{2}^2 + \|c_0\|_2^2.
      \end{aligned}
      \end{equation*}
      Thanks to the non-negativity of $b$, we get
      \begin{equation}\label{c1}
          \int_0^{\infty}\|\nabla a(s)\|_2^2ds + \int_0^{\infty}\|\nabla c(s)\|_2^2ds < +\infty.
      \end{equation}
      Thanks to the smoothness of classical solutions, we have for any $t_0>0$ $(a, b, c) \in (C^{1}([t_0, \infty); H^1(\Omega)))^3$. This implies that $\|\nabla a(\cdot)\|_2^2, \|\nabla c(\cdot)\|_2^2 \in C^1([t_0, \infty); L^2(\Omega))$. Combining this with \eqref{c1} gives
      \begin{equation*}
          \lim_{t\to\infty}\bra{\|\nabla a(s)\|_{2}^2 + \|\nabla c(s)\|_{2}^2} = 0.
      \end{equation*}
      This implies that there are constants $a^*\ge 0$ and $c^*\ge 0$ such that 
      \begin{equation*}
          \lim_{t\to \infty}\bra{\|a(t) - a^*\|_2+\|c(t) - c^*\|_2} = 0.
      \end{equation*}
      Thanks to the $L^{\infty}(\Omega)$-boundedness and H\"older's inequality, we have for any $p \in (2,\infty)$,
      \begin{equation}
          \label{Lp_convergence} \lim_{t\to\infty}\bra{\|a(t) - a_\infty\|_{p} + \|c(t) - c_\infty\|_{p}} = 0.
      \end{equation}
      Denote by $\{S_a(t) = e^{d_1\Delta t}\}_{t\ge 0}$  the semigroup generated by the operator $d_1\Delta$ with homogeneous Neumann boundary condition. We then have the estimate
      \begin{equation*}
          \|S_a(t)f\|_{{\infty}} \le C\bra{1 + t^{-n/(2p)}}\|f\|_{p} \quad \forall t>0,
      \end{equation*}
      where $C$ is independent of $t$ and $f$. Then, by using the Duhamel formula, we have
      \begin{equation*}
          a(t+1) = S_a(1)a(t) + \int_0^1S_a(1-s)b(t+s)(c(t+s)-a(t+s))ds 
      \end{equation*}
      and thus
      \begin{equation*}
      \begin{aligned}
          \|a(t+1)-a^*\|_{{\infty}} &\le \|S_a(1)(a(t) - a^*)\|_{{\infty}} + \int_0^1\|S_a(1-s)b(t+s)(c(t+s)-a(t+s))\|_{{\infty}}\ ds\\
          &\le C\|a(t) - a^*\|_{{p}} + C\mathcal K\int_0^1\bra{1+(1-s)^{-n/(2p)}}\|c(t+s) - a(t+s)\|_{p}\ ds.
      \end{aligned}
      \end{equation*}
      By choosing $p>n/2$, sending $t\to\infty$, and using \eqref{Lp_convergence}, we obtain
      \begin{equation}\label{Linf_convergence_a}
          \lim_{t\to\infty}\|a(t) -a^*\|_{\infty} = 0.
      \end{equation}
      Similarly,
      \begin{equation}\label{Linf_convergence_c}
          \lim_{t\to\infty}\|c(t) -c^*\|_{\infty} = 0.
      \end{equation}
      From the conservation laws we get
      \begin{equation*}
          a^* + c^* = M_1 \quad \text{ and } \quad c^* + \lim_{t\to\infty}\int_{\Omega}b(x,t)dx = M_2.
      \end{equation*}
      We consider two cases:
      \begin{itemize}
          \item If $a^* = c^*$, then 
          \begin{equation*}
              \lim_{t\to\infty}\int_{\Omega}b(x,t)dx = M_2 - c^*  = M_2 - \frac{M_1}{2} \le 0.
          \end{equation*}
          Since $b(x,t) \ge 0$ a.e. $x\in\Omega$, it follows that
          \begin{equation*}
              \lim_{t\to\infty}\|b(t)\|_1 = \lim_{t\to\infty}\int_{\Omega}b(x,t)dx = 0.
          \end{equation*}
          Using the uniform-in-time boundedness of $b(t)$ in $L^\infty(\Omega)$, we obtain immediately
          \begin{equation*}
              \lim_{t\to\infty}\|b(t)\|_2 = 0.
          \end{equation*}
          \item If $a^* \ne c^*$, then from $M_1 \ge 2M_2$
          \begin{equation*}
              a^* \ge c^* + 2\lim_{t\to\infty}\int_{\Omega}b(x,t)dx \ge c^*,
          \end{equation*}
          and thus $a^* > c^*$. Thanks to \eqref{Linf_convergence_a} and \eqref{Linf_convergence_c}, there exists $\hat{t} > 0$ large enough such that
          \begin{equation*}
              a(x,t) - c(x,t) \ge \frac{a^* - c^*}{2} > 0, \quad \forall t\ge \hat{t}, \forall x\in\Omega.
          \end{equation*}
          Therefore, for $t\ge \hat{t}$,
          \begin{equation*}
              b_t - d_3\Delta b = b(c-a) \le -\frac{a^*-c^*}{2}b.
          \end{equation*}
          A standard argument leads to
          \begin{equation*}
              \lim_{t\to\infty}\|b(t)\|_2 = 0.
          \end{equation*}
      \end{itemize}
      In both cases, we always have
      \begin{equation*}
          \lim_{t\to\infty}\|b(t)\|_2 = 0.
      \end{equation*}
      From the conservation laws, obtain finally
      \begin{equation*}
          c^* = M_2 = \bar c_\infty \quad \text{ and } \quad a^* = M_1 - M_2 = \bar a_\infty,
      \end{equation*}
      which completes the proof of the Proposition.
  \end{proof}
  
  In the next result, we show that the boundary equilibrium is {\it locally exponentially} stable.
\begin{theorem} \label{thm5}
Fixed $M_1>2M_2.$ There exists a sufficiently small constant $\epsilon>0$ \blue{depending on $M_1, M_2, \mathcal K$, $\Omega$, $d_1, d_2, d_3$} such that 
    \[
    \norm{a_0-\bar a_\infty}_2^2+\norm{b_0}_2^2+\norm{c_0-\bar c_\infty }_2^2 \leq \epsilon,
    \]
then the the boundary equilibrium $(\bar a_\infty,0,\bar c_\infty)=(M_1-M_2,0,M_2)$ to the problem \eqref{P1} is locally  exponentially asymptotically stable. That means, there are positive constants $C$ and $\alpha_3$ such that for all $t\geq 0,$ the following estimate holds
\begin{equation*}
     \norm{a(t)-\bar a_\infty}_2^2+\norm{b(t)}_2^2+\norm{c(t)-\bar c_\infty }_2^2 \leq C e^{-\alpha_3 t}  .
\end{equation*}
\end{theorem}
\begin{proof}
Let us consider the following Lyapunov function 
\[
\mathcal{L}_1(t)=\norm{u_1(t)}_2^2 +  \beta \norm{u_2(t)}_2^2+\norm{u_3(t)}_2^2,
\]
where the constant $\beta$ is large enough and chosen later.
Along the trajectory of \eqref{s1}, we have
\begin{equation}\label{s6}
    \begin{aligned}
        \dfrac{1}{2}\dfrac{d\mathcal{L}_1}{dt}
        &=-\left(d_1\norm{\nabla u_1}_2^2 + d_2\beta\norm{\nabla u_2}_2^2+d_3\norm{\nabla u_3}_2^2   
        + \lambda \int_\Omega (u_1u_2+\beta u_2^2 -u_2u_3)dx \right)\\
     &  + \int_\Omega  u_2(u_3-u_1)(u_1+u_2-u_3)    dx  
       :=-\mathcal{I}_1 +\mathcal{I}_2.
    \end{aligned}
\end{equation}
To treat the first term $\mathcal{I}_1,$   regarding  to the Poincar\'e-Wirtinger inequality with   $\widetilde d_i =\dfrac{d_i}{ P(\Omega)}, \ 1\leq i\leq 3,$  one has
 \begin{align*}
 d_1\norm{\nabla u_1}_2^2 + d_2\beta\norm{\nabla u_2}_2^2+d_3\norm{\nabla u_3}_2^2 & \geq \widetilde d_1(\norm{  u_1}_2^2 - \bar u_1^2)+\widetilde d_2\beta(\norm{  u_2}_2^2 - \bar u_2^2)+\widetilde d_3(\norm{  u_3}_2^2 - \bar u_3^2)\\
 &\geq \widetilde d_1\norm{  u_1}_2^2 +\widetilde d_2\beta\norm{  u_2}_2^2+\widetilde d_3\norm{  u_3}_2^2  - (\widetilde d_1 +\beta \widetilde d_2 +\widetilde d_3 )\bar u_2^2 \\
 & \geq \widetilde d_1\norm{  u_1}_2^2 +\widetilde d_2\beta\norm{  u_2}_2^2+\widetilde d_3\norm{  u_3}_2^2  - (\widetilde d_1 +\beta \widetilde d_2 +\widetilde d_3 )\int_\Omega {u}_2^2 dx,
 \end{align*}
 where we applied the conservation laws \eqref{s2} in the second estimate and the Cauchy–Schwartz inequality in the last estimate. Consequently, for any $\mu \in(0,1),$ we obtain
 \begin{equation}\label{s3}
     \begin{aligned}
         \mathcal{I}_1& \geq \widetilde d_1\norm{  u_1}_2^2 +\widetilde d_2\beta\norm{  u_2}_2^2+\widetilde d_3\norm{  u_3}_2^2 +\int_\Omega \left( \lambda u_1u_2+(\lambda\beta  -\widetilde d_1 -\beta \widetilde d_2 -\widetilde d_3 )u_2^2 - \lambda u_2u_3 \right)dx \\
         &\geq(1-\mu) \left(  \widetilde d_1\norm{  u_1}_2^2 +\widetilde d_2\beta\norm{  u_2}_2^2+\widetilde d_3\norm{  u_3}_2^2 \right) 
+ \int_\Omega \left ( \mu \widetilde d_1 u_1^2  + \lambda u_1u_2+\dfrac{\lambda\beta}{4}  \widetilde d_3 u_2^2   \right) dx \\
&+ \int_\Omega \left (  \mu\widetilde d_3   u_3^2   - \lambda u_2u_3  + \dfrac{\lambda\beta}{4}  \widetilde d_3 u_2^2 \right) dx +\int_\Omega   \left( \beta(\mu\widetilde d_2 -\widetilde d_2 +\lambda /2 )  -\widetilde d_1   -\widetilde d_3 \right)    u_2^2 dx.
     \end{aligned}
 \end{equation}
Since $(\mu\widetilde d_2 -\widetilde d_2 +\lambda /2) \nearrow \lambda/2$ as $\mu \nearrow 1,$ we can choose $\mu \in (0,1)$ such that $\mu\widetilde d_2 -\widetilde d_2 +\lambda /2 >0.$
Then, choosing $\beta $ sufficiently large, for instance,
\[
\beta \geq \max \left\{\dfrac{\lambda}{\mu \widetilde d_1}, \dfrac{\lambda}{\mu \widetilde d_3}, \dfrac{\widetilde d_1   + \widetilde d_3}{\mu\widetilde d_2 -\widetilde d_2 +\lambda /2} \right\},
\]
 it reality follows from \eqref{s3} that
\begin{equation}\label{s4}
    \mathcal{I}_1 \geq (1-\mu) \left(  \widetilde d_1\norm{  u_1}_2^2 +\widetilde d_2\beta\norm{  u_2}_2^2+\widetilde d_3\norm{  u_3}_2^2 \right) .
\end{equation}
 This estimate confirms that the linear part of \eqref{s1}, after a suitable scaling, has a spectral gap. For controlling the nonlinear part of \eqref{s1}, we use the uniform boundedness of the solution. Indeed, applying the Cauchy–Schwarz inequality and the interpolation inequality, one has
		\begin{equation} \label{s5}
\begin{aligned}
    \mathcal{I}_2&\leq  C \sum_{i=1}^3 \norm{u_i}_3^3  
    \leq  C \sum_{i=1}^3 \norm{u_i}^{9/4 }_{2}\norm{u_i}_{6}^{3/4}\\
    &\leq  C   \sum_{i=1}^3  \norm{u_i}^{9/4 }_{2}  \leq  C ( \sum_{i=1}^3  \norm{u_i}^{2 }_{2})  (\sum_{i=1}^3 \norm{u_i}^{1/4} _{2})\\
    &\leq C( \sum_{i=1}^3  \norm{u_i}^{2 }_{2})  (\sum_{i=1}^3 \norm{u_i}^2 _{2})^{1/8},
\end{aligned}
		\end{equation}
  where we used Jensen's inequality in the last estimate. Combining \eqref{s6}, \eqref{s4} and \eqref{s5}, one gets
\begin{align*}
			\dfrac{ d\mathcal{L}_1(t)}{dt}& \leq -2  \alpha_3  \mathcal{L}_1(t) +  C  \mathcal{L}_1(t) ( \mathcal{L}_1(t)) )^{1/8 }  \\
			& \leq -\alpha_3  \mathcal{L}_1(t)  +  \mathcal{L}_1(t) \left( C ( \mathcal{L}_1(t) )^{1/8 } -\alpha_3 \right),
		\end{align*}
  where $\alpha_3 = (1-\mu)\min \{ \widetilde d_1, \beta \widetilde d_2, \widetilde d_3 \}.$
  Therefore, when $\mathcal{L}_1 (0)$ small enough, $\mathcal{L}_1(\cdot)$ is decreasing and thus $$\mathcal{L}_1(t) \leq C e^{- \alpha_3 t} , \text{ for all } t \geq 0,$$ by Gronwall inequality. The proof is completed.
\end{proof}

\subsection{Instability of boundary equilibrium - Bootstrap instability}\label{subsection:bootstrap}
In the case $M_2\leq M_1 <2M_2$, we will prove that the positive equilibrium is exponentially stable, while the boundary equilibrium is will be shown to be unstable. The first part is contained in the following theorem.
\begin{theorem}\label{thm0}
		Let $M_2\leq M_1<2M_2 $ and   the assumptions \eqref{A1}-\eqref{A2} hold. Then, the positive equilibrium is locally exponentially stable, that means,  there is a sufficiently small constant  $\epsilon>0$ such that   if 
  \[
  \norm{a_0-a^*_\infty}^2_2+\norm{b_0-b^*_\infty}^2_2+\norm{c_0-c^*_\infty}^2_2 \leq \epsilon,
  \]
  then 
  \begin{equation*}
      \norm{a(t)-a^*_\infty}^2_2+\norm{b(t)-b^*_\infty}^2_2+\norm{c(t)-c^*_\infty}^2_2 \leq C e^{-\alpha_2 t} \text{ for all }t\geq 0,
  \end{equation*}
  for some positive constants $C, \alpha_2$ which are independent of $t$.
	\end{theorem}
\begin{proof}
Following the ideas from Theorem \ref{thm2}, we will show that $b$ is bounded from below in a measurable set with positive measure. We first prove that
  \begin{equation}\label{ep5}
      \norm{a(t)-a^*_\infty}^2_2+\norm{b(t)-b^*_\infty}^2_2+\norm{c(t)-c^*_\infty}^2_2 \leq C\epsilon  ,
  \end{equation}
holds for all $t\geq 0,$ where $C$ is a positive constant.  

  Indeed, multiplying the first equation and the third equation in \eqref{ep2} by $w_1$ and $w_3$ respectively, then   integrating over $\Omega,$ we arrive at
\[
\dfrac{d}{dt}   \left ( \norm{w_1}_2^2 + \norm{w_3}_2^2 \right) +d_1 \norm{\nabla w_1}_2^2+d_3\norm{\nabla w_3}_2^2 = \int_\Omega -(w_2+b^*_\infty)(w_3-w_1)^2dx\leq 0.
\]
Therefore,
\[
\dfrac{d}{dt}   \left ( \norm{w_1}_2^2 + \norm{w_3}_2^2 \right) \leq 0,
\]
 i.e., $ \norm{w_1(t)}_2^2 + \norm{w_3(t)}_2^2 $ is decreasing with respect to $t,$ and hence, for all $t\geq 0,$
 \begin{equation}\label{ep3}
     \norm{w_1(t)}_2^2 + \norm{w_3(t)}_2^2 \leq \norm{w_1(0)}_2^2 + \norm{w_3(0)}_2^2 \leq \norm{a_0-a^*_\infty}^2_2 +\norm{c_0-c^*_\infty}^2_2  \leq \epsilon.
 \end{equation}

 To control  $\norm{w_2(t)}_2 ,$ we do similar in the proof of Theorem \ref{thm2} the remark that, in this case, we have  \eqref{ep3} instead of \eqref{t2}. Therefore, we get an estimate similar to \eqref{t0}
\begin{equation}\label{t11'}
  \dfrac{d}{dt}\norm{w_2}_2^2 \leq  -\kappa_3 \norm{w_2}_2^2   +C   \epsilon ,
\end{equation}
where $\kappa_3 $ and $C$ are positive constant. The classical Gronwall lemma gives
\begin{equation}\label{t12'}
\begin{aligned}
\norm{w_2(t)}^2_2 \leq \norm{w_2(0)}^2_2 e^{-\kappa_3 t}+\dfrac{C \epsilon}{\kappa_3} \le  \max\left (1+ \dfrac{C  }{\kappa_3}  \right)\epsilon.
\end{aligned}
\end{equation}
 Combining \eqref{ep3} and  \eqref{t12'}, we deduce \eqref{ep5}.
 
\medskip
Fix $t>0$. We claim that for $\epsilon$ sufficiently small, there are positive constants $\delta>0$ and $\theta\in (0,1)$, and a positive measurable set $\omega \subset\Omega$  with $|\omega| \geq 1-\theta$ such that $b(x,t)\geq \delta$ for all $x\in \omega$. It should be emphasized that {\it $\delta$ and $\theta$ are independent of $t$}. Indeed, for $\epsilon$ small enough such that $\sqrt{C\epsilon}< (b^*_\infty)^2$, we can choose $\kappa>0$ and $\theta \in (0,1)$ such that \[
 {\sqrt{C\epsilon}} < \theta\kappa ^2 <\kappa ^2 <(b^*_\infty)^2.
\]
Then, applying  Chebyshev's inequality   into account and using \eqref{ep5}, we have
\[
\left| \{  x\in\Omega: |b(x,t)-b^*_\infty| \geq \kappa  \} \right| \leq \dfrac{1}{\kappa^2}\norm{b( t)-b^*_\infty}_2\leq \dfrac{\sqrt{C\epsilon}}{\kappa^2},
\]
which leads to
\[
\left| \{  x\in\Omega: b^*_\infty -\kappa \leq b(x,t)\leq b^*_\infty +\kappa  \} \right| \geq 1 - \dfrac{\sqrt{C\epsilon}}{\kappa^2} \ge 1-\theta.
\]
Choose $\delta =b^*_\infty -\kappa >0$ and $\omega =\{ x\in \Omega: b(x,t)\geq  \delta,\forall t\geq t_0 \},$ then \[ |\omega|\geq 
 \left| \{  x\in\Omega: b^*_\infty -\kappa \leq b(x,t)\leq b^*_\infty +\kappa  \} \right| \geq 1 
 - \theta.\]
Therefore, the claim is proved. Finally, the remainder of the proof proceeds in the same manner as the proof of Theorem \ref{thm2}, so we omit it here.
\end{proof}

\medskip
We now turn to the instability of the boundary equilibrium $( \bar a_\infty,0,\bar c_\infty)=(M_1-M_2,0,M_2)$. In order to do that, we employ the bootstrap instability scheme of Guo and Strauss \cite{Gou3}, \cite{Gou1},\cite{Gou2}.  

 \medskip
We recall as in Section \ref{subsection:qualitative}, the new variables as perturbation around the boundary equilibria $u=(a-\bar  a_\infty,b,c-\bar c_\infty),$ then
	  \eqref{p1} can be rewritten as following
	\begin{equation}\label{p2}
		\begin{cases}
			\partial_t u_1 -d_1\Delta u_1= (\bar c_\infty-\bar a_\infty)u_2+u_2(u_3-u_1), &\text { in } \Omega \times\mathbb R_+,\\
			\partial_t u_2 -d_2\Delta u_2= (\bar c_\infty-\bar a_\infty)u_2+u_2(u_3-u_1), &\text { in } \Omega \times\mathbb R_+,\\
			\partial_t u_3-d_3\Delta u_3=-(\bar c_\infty-\bar a_\infty)u_2-u_2(u_3-u_1),  &\text { in } \Omega \times\mathbb R_+,\\
			\nabla u_i \cdot \nu =0, \ 1\leq i\leq 3, &\text { on } \partial\Omega \times\mathbb R_+,\\
			u(x,0)=u_{0}(x)=(a_0-\bar a_\infty,b_0,c_0-\bar c_\infty), &\text { in } \Omega .
	\end{cases}\end{equation}
	System \eqref{p2} can be expressed as 
 \begin{equation}\label{p3}
      u_t=L_1u+N_1(u), 
 \end{equation}
	where $$L_1=  \begin{pmatrix} 
		d_1\Delta & \bar c_\infty-\bar a_\infty & 0 \\
		0 & d_2\Delta +\bar c_\infty-\bar a_\infty & 0\\
		0 & -(\bar c_\infty-\bar a_\infty) & d_3\Delta   \\
	\end{pmatrix} \quad \text{ and } \quad N_1(u)=\begin{pmatrix} 
		-u_1u_2+u_2u_3  \\
		-u_1u_2+u_2u_3\\
		u_1u_2-u_2u_3\\
	\end{pmatrix}.$$
Taking derivative with respect to time, we get from \eqref{p3} the equation
 \begin{equation} \label{p4} u_{tt}=Lu_t+\partial_t\left (N_1(u)\right ).
 \end{equation}
 Denote $y=(u,u_t)^T,$ then $y$ satisfies the following system 
 \begin{equation}\label{p5}
     y_t=Ly+N(y),
 \end{equation}
 where  $L=  \begin{pmatrix} 
		L_1 & 0 \\
		0 & L_1\\
	\end{pmatrix}$
	and $N(y)=\begin{pmatrix} 
		N_1(u)  \\
		\partial_t\left (N_1(u)\right )\\
	\end{pmatrix}.$

	Consider the following norms  \[
\norm{y}=\norm{u}_2+\norm{u_t}_2;\  \vertiii{y}   = \norm {u}_{H^2(\Omega)}+\norm{u_t}_{2} .
    \] 
    We now can state the main theorem in this section.
	\begin{theorem}\label{thm4} Let $\Omega \subset \mathbb R^n,$ for $n\in\{1,2,3\}$, and assume that  \eqref{A1}-\eqref{A2} hold. Assume additionally \textcolor{red}{$M_1 \ge 2M_2$}. \blue{Then there exist positive constants $\theta_0, \tau_0$ depending only on $M_1, M_2, d_1, d_2, d_3, \Omega, \mathcal{K}$ and $ \int_{\Omega}b_0(x)dx$ such that if $y(0) = \delta y_0$ with $\|y_0\| = 1$ and $\delta$ sufficiently small, then
    \begin{equation*}
        \|y(t)- \delta e^{Lt}y_0\| \le C\delta^2e^{2(\bar c_\infty - \bar a_\infty)t},
    \end{equation*}
    for all
    \begin{equation*}
        0\leq t\leq T^\delta:=\dfrac{1}{(\bar c_\infty - \bar a_\infty)}\ln \dfrac{\theta_0}{\delta},
    \end{equation*}
    and, at the escape time $T^\delta$, it holds
    \begin{equation*}
        \|y(T^\delta)\| \ge \tau_0 > 0.
    \end{equation*}
    }
	\end{theorem}
 \blue{For the convenience of the reader, we briefly outline the proof of   Theorem \ref{thm4}  here. Firstly, based on the construction of the higher order norm and the lower order norm for the solution (and the fact that the higher order norm of the solution can be controlled at small times by the low norm), the detailed energy (in high norm) estimates for solutions to the nonlinear problem are provided in Lemma \ref{lm4} and Lemma \ref{lm5}. Then, some   estimates of the solution in lower order norm to the  linearized system are successfully  established in Lemma \ref{lm7} and Lemma \ref{lm6}. Finally, we demonstrate that the  nonlinear solution is sufficiently close to the linear solution, and for these reasons, the existence of escape time is completely confirmed in Theorem \ref{thm4}. In this case, $\bar c_\infty -   \bar a_\infty = 2M_2 - M_1 >0,$ this fact plays a crucial role in our approach.  We also note here that in all the following lemmas, the assumptions in Theorem \ref{thm4} hold.}
	\begin{lemma}\label{lm4}
		For $\vertiii{y}<\sigma = \dfrac{1}{C_{SI}}$ with $C_{SI}$ is the Sobolev embedding constant in 
$H^2(\Omega)\hookrightarrow  L^\infty(\Omega),$ i.e., $\norm{u}_\infty \leq C_{SI}\norm{u}_{H^2(\Omega)},$ for all $u\in H^2(\Omega).$ Then,  there exists $C_\sigma>0$ \blue{depending only on $M_1, M_2, \Omega, \mathcal{K}$} such that the following estimate holds
		 \begin{equation*}
		     \vertiii{y}^2\leq C_\sigma \left( \int_0^t \norm{y}^2ds+\norm{y(0)}^2 \right) \quad \forall t\ge 0.
		 \end{equation*}
	\end{lemma}	
	\begin{proof}
 By the assumptions, we first see that $\norm{u_i}_\infty\leq   C_{SI}\norm{u_i}_{H^2(\Omega)}\leq \sigma C_{SI}\leq 1, $ for all $1\leq i\leq 3.
 $
Next, multiplying both sides of \eqref{p2}  by $u_1,u_2,u_3$ respectively,   integrating with respect to $x $ over $\Omega,$ then    applying Cauchy–Schwarz inequality inequality and the boundedness of $u_i,$ yields,
\begin{align*}
    &\dfrac{1}{2}\dfrac{d}{dt} \left( \norm{u_1}^2_2+\norm{u_2}^2_2+\norm{u_3}^2_2  \right)+\left( d_1\norm{\nabla u_1}^2_2+d_2\norm{\nabla u_2}^2_2+d_3\norm{\nabla u_3}^2_2  \right) \\
    &=\int_\Omega(\bar c_\infty -\bar a_\infty) ( u_1u_2+u_2^2-u_2u_3)dx 
    +\int_\Omega u_2(u_3-u_1)(u_1+u_2-u_3)dx
    \\
    &\leq (\bar c_\infty -\bar a_\infty)\int_\Omega ( \frac{1}{2} u_1^2 +2 u_2^2+ \frac{1}{2}   u_3^2)dx  + 3 \int_\Omega (u_1^2+u_2^2+u_3^2 )dx  \\
    & \leq (3+2(\bar c_\infty -\bar a_\infty)) \norm u_2^2,
\end{align*}
which implies that
\begin{equation}\label{eq19}
    \dfrac{d}{dt}   \norm{u }^2_2    \leq (6+4(\bar c_\infty -\bar a_\infty)) \norm{u }^2_2 .
\end{equation}
Similarly, by Holder's inequality and the fact that  $\norm {u_i}_\infty\leq 1,$ we obtain from \eqref{p4} that
\begin{align*}
    &\dfrac{1}{2}\dfrac{d}{dt} \left( \norm{\partial_t u_1}^2_2+\norm{\partial_t u_2}^2_2+\norm{\partial_t u_3}^2_2  \right)+\left( d_1\norm{\nabla \partial_t u_1}^2_2+d_2\norm{\nabla \partial_t u_2}^2_2+d_3\norm{\nabla \partial_t u_3}^2_2  \right) \\
    &=\int_\Omega(\bar c_\infty -\bar a_\infty) ( \partial_t u_1 \partial_t u_2+(\partial_t u_2)^2-\partial_t u_2\partial_t u_3)dx 
    \\
    &+\int_\Omega \left[u_2(\partial_t u_3-\partial_t u_1)+ \partial_t u_2 (u_3-u_1) \right](\partial_t u_1+\partial_t u_2-\partial_t u_3)dx
    \\
    &\leq  \left (6+ 2 (\bar c_\infty - \bar a_\infty) \right )  \norm{\partial_t u}_2^2.
    \end{align*}
Thus,
\begin{equation}\label{eq20}
    \dfrac{d}{dt} \norm{\partial_t u}^2_2 \leq (12+ 4 (\bar c_\infty - \bar a_\infty)) \norm{\partial_t u}^2_2.
\end{equation}
 For any $t>0,$ we rewrite system \eqref{p3} as an elliptic system $$ L_1u=u_t-N_1(u),$$
 i.e.,
 \begin{equation}\label{g11}
         \begin{cases}
			d_1\Delta u_1 +  (\bar c_\infty-\bar a_\infty)u_2= \partial_t u_1 - u_2(u_3-u_1), &\text { in } \Omega \times\mathbb R_+,\\
			 d_2\Delta u_2+ (\bar c_\infty-\bar a_\infty)u_2=\partial_t u_2-u_2(u_3-u_1), &\text { in } \Omega \times\mathbb R_+,\\
			 d_3\Delta u_3 -(\bar c_\infty-\bar a_\infty)u_2=\partial_t u_3 -u_2(u_3-u_1),  &\text { in } \Omega \times\mathbb R_+.\\
     \end{cases}
 \end{equation}
Taking the $L^2(\Omega)$-norm of both sides of the first equation in \eqref{g11} gives
\begin{equation*}
    d_1\|\Delta u_1\|_2^2 \le 3(\bar c_\infty - \bar a_\infty)^2\|u_2\|_2^2 + 3\|\partial_t u_1\|_2^2 + 3\|u_2(u_3- u_1)\|_2^2.
\end{equation*}
Repeating this for the second and the third equation in \eqref{g11}, there exists a constant $C$ such that  for all $1\leq i \leq 3,$
\begin{align}\label{eq17.2}
    \norm{u_i}_{H^2(\Omega)}^2 & \leq C\left(\norm{u_i}_2^2+ \norm{ \Delta u_i}_2^2 \right)\leq C \left ( \norm u_2^2+\norm{u_t}_2^2+\norm {N_1(u)}_2^2 \right )\\\notag
    &\leq C \left (  \norm u_2^2+\norm{u_t}_2^2 +\sum\limits_{1\leq i\neq j \leq 3} \norm{u_iu_j}_2^2 \right) \\ \notag
    &\leq   C \left (  \norm u_2^2+\norm{u_t}_2^2  \right),
\end{align}
  where the last inequality holds due to the uniform boundedness of $u_i$. Substituting  inequalities \eqref{eq19}  and \eqref{eq20} into \eqref{eq17.2}, one obtains
 \begin{align}\label{eq21}
     \norm{u_i}_{H^2(\Omega)}^2&\leq C \left( \int_0^t \dfrac{d}{ds}\norm {u(\cdot,s)}_2^2ds+\norm{u_0}^2_2+\int_0^t \dfrac{d}{ds}\norm{\partial_tu(\cdot,s)}_2^2 ds + \norm{u_t(0)}_2^2 \right) \notag\\
     &\leq C  \left(  \int_0^t ( \norm{u}_2^2+\norm{u_t}_2^2) ds + \norm{u_0}^2_2+\norm{u_t(0)}_2^2 \right) \notag\\
     &\leq C \left( \int_0^t \norm {y}^2ds+\norm{y(0)}^2 \right).
 \end{align}
 Now, from \eqref{eq20}, we can estimate the second part of $\vertiii{y}.$
\begin{equation}\label{eq22}
    \norm{u_t}_2^2 =\int_0^t \dfrac{d}{ds} \norm{u_t(\cdot,s) }_2^2ds+ \norm{u_t(0)}_2^2 \leq C \left( \int_0^t \norm{u_t}^2_2ds+\norm{u_t(0)}_2^2 \right) \leq C \left( \int_0^t \norm{y}^2ds+\norm{y(0)}^2 \right).
\end{equation}
 Combining \eqref{eq21} and \eqref{eq22}, one can estimate that $$
\vertiii{y}^2\leq C \left( \int_0^t \norm{y}^2ds+\norm{y(0)}^2 \right).$$  
This completes the proof of Lemma \ref{lm4}.
	\end{proof}
Next, we can demonstrate the following estimate for nonlinear part.
\begin{lemma}\label{lm5}
    There exists a constant $C_N$ \blue{depending only on $\Omega$} such that 
    \begin{equation}\label{g1}
        \norm{N(y)}= \norm{N_1(u)}_2+\norm{\partial_t (N_1(u))}_2\leq C_N   \vertiii{y}^2  .
    \end{equation}
\end{lemma}
\begin{proof}
Applying  Sobolev embedding inequality $\norm{u_i}_\infty\leq C_{SI} \norm{u_i}_{H^2(\Omega)}$ and Cauchy–Schwarz inequality inequality, we estimate the first term in the left hand side of \eqref{g1}
\begin{equation}\label{g9}
    \begin{aligned}
        \norm{N_1(u)}_2&=   3\norm{u_2(u_3-u_1)}_2\\
	&\leq 3\norm{u_2}_\infty\norm{u}_2\\
        &\leq 3C_{SI}  \norm{u}_{H^2(\Omega)} \norm{u}_2 \leq \dfrac{3}{4}C_{SI} \vertiii{y}^2.
    \end{aligned}
\end{equation}
 Likewise, the second term in    the left hand side of \eqref{g1} can be controlled by
 \begin{equation}\label{g10}
    \begin{aligned}
        \norm{\partial_t (N_1(u))}_2=3 \norm{\partial_t(u_2(u_3-u_1))}_2  
	 \leq 12 \norm{ u}_\infty \norm{\partial_tu}_2 
        \leq 3C_{SI}  \vertiii{y}^2.
    \end{aligned}
\end{equation}
Combining \eqref{g9} and \eqref{g10}, we see that \eqref{g1} is verified.
\end{proof}
	\begin{lemma}\label{lm7}
		For any $\norm{y_0}=1$ with $C_1:= \int_\Omega {y}_{0,2}(x) dx \ne 0,$ there exists a positive constant $C_{P}$ \blue{depending only on $M_1, M_2, C_1$} such that \[
  \norm {e^{Lt}y_0}  \geq C_P e^{(\bar c_\infty-\bar a_\infty)t}, \] for all $t>0.$
	\end{lemma}
	\begin{proof}
		Let us consider the linear system
		\begin{equation}\label{eq9p}
		\begin{cases}
			\partial_t \wh{u}_1 -d_1\Delta \wh u_1= (\bar c_\infty-\bar a_\infty)\wh u_2,  &\text { in } \Omega \times\mathbb R_+ ,\\
			\partial_t \wh u_2 -d_2\Delta \wh u_2= (\bar c_\infty-\bar a_\infty)\wh u_2,  &\text { in } \Omega \times\mathbb R_+ ,\\
			\partial_t \wh u_3-d_3\Delta \wh u_3=-(\bar c_\infty-\bar a_\infty)\wh u_2 ,  &\text { in } \Omega \times\mathbb R_+,\\
			\nabla \wh u_i \cdot \nu =0, 1\leq i\leq 3,  &\text { on } \partial\Omega \times\mathbb R_+,\\
			\wh u(x,0)=y_0,  &\text { in } \Omega .
	\end{cases}\end{equation}
		Taking the integration over $\Omega,$ one obtains form \eqref{eq9p} that
		\begin{equation}\label{l1p}
			\begin{cases}
				\frac{d}{dt} \int_\Omega \wh u_1(x,t) dx =(\bar c_\infty-\bar a_\infty)\int_\Omega \wh u_2(x,t) dx,\\
				\frac{d}{dt} \int_\Omega \wh u_2(x,t) dx=(\bar c_\infty-\bar a_\infty)\int_\Omega \wh u_2(x,t) dx\\
				\frac{d}{dt} \int_\Omega \wh u_3(x,t) dx=-(\bar c_\infty-\bar a_\infty)\int_\Omega \wh u_2(x,t) dx,\\
		\end{cases}\end{equation}  
		Therefore,  we can compute
		\begin{equation*}
			\begin{cases}
				\int_\Omega \wh u_1(x,t) dx =C_1 e^{(\bar c_\infty-\bar a_\infty)t}-C_1+C_2,\\
				\int_\Omega \wh u_2(x,t) dx=C_1e^{(\bar c_\infty-\bar a_\infty)t}  ,\\
				\int_\Omega \wh u_3(x,t) dx= -C_1 e^{(\bar c_\infty-\bar a_\infty)t}+C_1+C_3 ,\\
		\end{cases}\end{equation*} 
		where the constants $C_1= \int_\Omega y_{0,2}(x) dx,\ C_2= \int_\Omega  y_{0,1}(x) dx$ and $C_3=\int_\Omega  y_{0,3}(x) dx.$
		Thus,
		\begin{align*}
			\norm{e^{Lt}y_0}_2^2&= \int_\Omega \left( |\wh u_1(x,t)|^2+|\wh u_2(x,t)|^2+|\wh u_3(x,t)|^2\right) dx\notag \\ 
			& \geq   \left (\int_\Omega \wh u_1(x,t)dx \right)^2+\left (\int_\Omega \wh u_2(x,t)dx \right)^2+\left(\int_\Omega \wh u_3(x,t)dx \right)^2\notag \\
			& \geq \left(\int_\Omega \wh u_2(x,t)dx \right)^2  = C_1^2e^{2(\bar c_\infty-\bar a_\infty)t}  
		\end{align*}
		  Hence,
		\begin{equation}\label{l3p}
		    \norm{e^{Lt}y_0}_2\geq  { |C_1|}  e^{(\bar c_\infty-\bar a_\infty)t}.
		\end{equation}  
Now we turn to estimate  the second part $\wh u_{tt}=L \wh u_t.$ From \eqref{l1p}, we obtain
\begin{equation*} 
			\begin{cases}
				\frac{d}{dt} \int_\Omega \wh u_1(x,t) dx =C_1(\bar c_\infty-\bar a_\infty)e^{(\bar c_\infty-\bar a_\infty)t},  \\
				\frac{d}{dt} \int_\Omega \wh u_2(x,t) dx=C_1(\bar c_\infty-\bar a_\infty)e^{(\bar c_\infty-\bar a_\infty)t}, \\
				\frac{d}{dt} \int_\Omega \wh u_3(x,t) dx=-C_1(\bar c_\infty-\bar a_\infty)e^{(\bar c_\infty-\bar a_\infty)t}. 
		\end{cases}\end{equation*}  
  Then,
 \begin{align*}
      \norm{\partial_t \wh u}_2^2&=\int_\Omega\left ((\partial_t \wh u_1)^2+ (\partial_t \wh u_2)^2+(\partial_t \wh u_3)^2 \right) dx\\
      &\geq \dfrac{1}{3 } \left (\int_\Omega (\partial_t \wh u_1 +  \partial_t \wh u_2 + \partial_t \wh u_3 )  dx\right)^2\\
      &=\dfrac{|C_1|^2(\bar c_\infty-\bar a_\infty)^2e^{2(\bar c_\infty-\bar a_\infty)t}}{3 } .
 \end{align*}
By this,
\begin{equation}\label{l4p}
\norm{\partial_t(e^{Lt}y_0)}_2\geq \dfrac{ |C_1|(\bar c_\infty-\bar a_\infty)}{\sqrt{3 }} e^{(\bar c_\infty-\bar a_\infty)t}.
\end{equation}
It can be deduced from \eqref{l3p} and \eqref{l4p} that
  \[
  \norm {e^{Lt}y_0}  \geq C_P e^{(\bar c_\infty-\bar a_\infty)t},
  \]
  where the constant $C_P=\min\left\{{ |C_1|} ,\dfrac{ |C_1|(\bar c_\infty-\bar a_\infty)}{\sqrt{3 }}\right\}  $ depends on $C_1$ and $M_1, M_2$.
	\end{proof}
\begin{lemma}\label{lm6}
		Consider the linear partial equation \eqref{eq9p}. There exists a constant $C_L$ \blue{depending only on $\Omega$, $M_1, M_2$, $d_1$, $d_2$, $d_3$} such that $$\norm{ e^{Lt}}_{(L^2(\Omega),L^2(\Omega))} \leq C_Le^{(\bar c_\infty-\bar a_\infty) t}  .$$
	\end{lemma}
	\begin{proof}
	   Denote by $0=\lambda_0>\lambda_1\geq \lambda_2\geq \dots\geq \lambda_n\geq\dots \to -\infty$ the eigenvalues of the Neumann Laplacian,
     \begin{equation*}
         \begin{cases}
             \Delta v=\lambda v, \quad &x\in\Omega,\\
             \nabla v  \cdot \nu =0, & x\in\partial \Omega.
         \end{cases}
         \end{equation*}
        Then, for  an index $i$ fixed, eigenvalues $\eta$ of the linear operator $L$ defined by
\begin{equation*}
       \begin{vmatrix}
          d_1\lambda_i-\eta  & \bar c_\infty-\bar a_\infty & 0 \\
           0 & d_2\lambda_i +\bar c_\infty-\bar a_\infty - \eta& 0\\
           0 & \bar a_\infty-\bar c_\infty & d_3\lambda_i -\eta
       \end{vmatrix}
       =0.
\end{equation*}
That means, $\eta=d_1\lambda_i\leq 0,\ \eta=d_3\lambda_i\leq 0,\ \eta=d_2\lambda_i+\bar c_\infty-\bar a_\infty\leq \bar c_\infty-\bar a_\infty. $   
 Therefore,  the maximal eigenvalue of $L$ satisfies  $\eta_{\max} \leq \bar c_\infty-\bar a_\infty.$ That means, all eigenvalues of $L$ is bounded above by the positive number $ \bar c_\infty-\bar a_\infty$. Lemma \ref{lm6} is proved.
	\end{proof}
We  finish this section  by giving  the proof of Theorem \ref{thm4}.
	\begin{proof}[Proof of instability Theorem \ref{thm4}]
We choose
\begin{equation}\label{l2}
     \theta_0=\min\left\{\dfrac{\bar c_\infty - \bar a_\infty}{C_N\widetilde C^2 } ,\dfrac{C_P(\bar c_\infty - \bar a_\infty)}{2C_LC_N\widetilde C^2},\dfrac{\sigma}{\widetilde C } \right\},
\end{equation}
where $C_N$, $C_P$, $C_L$ are in Lemmas \ref{lm5}, \ref{lm7}, \ref{lm6}, respectively, and $\widetilde{C}$ is defined in \eqref{g6}. Let us denote 
\begin{align*}
    T^\delta &= \dfrac{1}{\bar c_\infty - \bar a_\infty}\ln \dfrac{\theta_0}{\delta},\\
    T^*&=\sup\limits_{t\geq 0}\{ \vertiii{y}<\sigma \},\\
    T^{**}&=\sup\limits_{t\geq 0} \left\{ \norm{y(t)}\leq 2C_L\delta e^{(\bar c_\infty - \bar a_\infty) t} \norm{y_0} \right\}.
\end{align*}
    For $t\leq \min\{T^\delta,T^*,T^{**} \},$ by the definition of $T^{**}$ and thanks to Lemma \ref{lm4}, we obtain
\begin{equation*}
   \begin{aligned}
     \vertiii{y(t)}^2&\leq C_\sigma \left( \int_0^t \norm{y(\cdot,s)}^2ds+\norm{y(0)}^2 \right)\\
     & \leq C_\sigma \left( \int_0^t  4C_L^2\delta^2  e^{2(\bar c_\infty - \bar a_\infty) s}ds+\delta^2 \right)\\
     &\leq C_\sigma\left(\dfrac{2}{\bar c_\infty - \bar a_\infty}C_L^2\delta^2e^{2(\bar c_\infty - \bar a_\infty) t}+\delta^2e^{2(\bar c_\infty - \bar a_\infty) t}\right).
\end{aligned} 
\end{equation*}
Hence, for all $t\le \min\{T^\delta, T^*, T^{**}\}$,
\begin{equation}\label{g6}
    \vertiii{y(t)}\leq \widetilde C \delta e^{(\bar c_\infty - \bar a_\infty) t},\text{ for } \widetilde C=\sqrt{ \left(\dfrac{2}{\bar c_\infty - \bar a_\infty}C_L^2+1\right)C_\sigma}.
\end{equation}
    Next, applying the Duhamel principle to the equations $y_t=Ly+N(y),$ we have
\begin{equation}\label{g7}
   \begin{aligned}
        \norm{y(t)-\delta e^{Lt}y_0}&=\norm{ \int_0^t e^{L(t-s)}N(y(\cdot,s))ds}\\
        &\leq \int_0^t \norm{  e^{L(t-s)}} \norm{ N(y(\cdot,s))}ds\\
        &\leq C_LC_N\int_0^t e^{(\bar c_\infty - \bar a_\infty)(t-s)} \vertiii{y(s)}^2ds\\
        &\leq \dfrac{ C_LC_N\widetilde C^2 \delta^2 }{\bar c_\infty - \bar a_\infty}e^{2(\bar c_\infty - \bar a_\infty) t},
   \end{aligned} 
\end{equation}
   where we used Lemma \ref{lm5} in the third inequality and  \eqref{g6} in the last inequality.

    \medskip
   Now we claim that $T^\delta= \min\{T^\delta,T^*,T^{**} \}$.  Indeed, suppose, ad absurdum, $T^{**}<T^\delta$  is the smallest, then $\delta e^{(\bar c_\infty - \bar a_\infty) T^{**}} <\delta e^{(\bar c_\infty - \bar a_\infty) T^\delta}=\theta_0.$  
   Thus, from  \eqref{g7}, we see that
   \begin{equation*}
     \begin{aligned}
\norm{y(T^{**})}&\leq \norm{y(T^{**})-\delta e^{LT^{**}}y_0}+\norm{\delta e^{LT^{**}}y_0}\\
&\leq \dfrac {C_LC_N\widetilde C^2 \delta^2 }{\bar a_\infty}  e^{2(\bar c_\infty - \bar a_\infty) T^{**}}+C_L\delta  e^{(\bar c_\infty - \bar a_\infty) T^{**}} \\
&<  C_L \delta  e^{(\bar c_\infty - \bar a_\infty) T^{**}}\left ( \dfrac{ C_N\widetilde C^2 \theta_0}{\bar c_\infty -\bar a_\infty} + 1\right)\\
&\leq 2  C_L \delta  e^{(\bar c_\infty - \bar a_\infty) T^{**}},
     \end{aligned}  
   \end{equation*}
   since $\theta_0\leq \dfrac{\bar c_\infty - \bar a_\infty}{C_N\widetilde C^2 },$ which contrasts with the definition of $T^{**}$. If $T^{*}<T^\delta$  is the smallest, we deduce from \eqref{g6} that 
  \begin{equation*}
\vertiii{y(T^*)} \leq \widetilde C \delta e^{(\bar c_\infty - \bar a_\infty) T^*} < \widetilde C \delta e^{(\bar c_\infty - \bar a_\infty) T^\delta}=\widetilde C\theta_0\leq \sigma,
   \end{equation*}
  due to  $\theta_0\leq \dfrac{\sigma}{\widetilde C },$ which contradicts the definition of $T^{*}.$
   
   Finally, we show that $T^\delta$ is the escape time. Indeed,
   \begin{equation*}
     \begin{aligned}
         \norm{y(T^\delta)}&\geq \norm{\delta e^{LT^\delta}y_0 } -  \norm{y(T^\delta)-\delta e^{LT^\delta}y_0}\\
         &\geq C_P\delta e^{(\bar c_\infty - \bar a_\infty) T^\delta}-\dfrac{C_LC_N\widetilde C^2\delta^2}{\bar c_\infty - \bar a_\infty}e^{2(\bar c_\infty - \bar a_\infty) T^\delta}\\
         &=C_P\theta_0 -\dfrac{C_LC_N\widetilde C^2}{\bar c_\infty - \bar a_\infty}\theta_0^2 \\
         &\geq \dfrac{1}{2}C_P\theta_0=\tau_0,
     \end{aligned}  
   \end{equation*}
   as $\theta_0 \leq \dfrac{C_P(\bar c_\infty -\bar a_\infty)}{2C_LC_N\widetilde C^2},$ where the constant $\tau_0$ does not depend on $\delta.$
\end{proof}
\section{Symmetric irreversible reactions}\label{sec:symmetric}

Similarly to Theorem \ref{thm:global-existence}, we have the following result for system \eqref{eq1'}.
\begin{theorem}\label{thm:global-existence-1}
    Assume \eqref{A1}--\eqref{A2}. Then there exists a unique global classical solution in the following sense: for any $T>0$,
    \begin{equation*}
        a,b,c\in C([0,T),L^p(\Omega) )\cap C^{1,2}((0,T)\times \overline{\Omega}), \quad \forall 1\le p < +\infty,
    \end{equation*}
  satisfying each equation in \eqref{eq1'} pointwise. Moreover, this solution is bounded uniformly in time, i.e. there is a positive constant $\mathcal{H}$ such that
  \begin{equation}\label{uit-boundedness_1}
 \norm{a(t)}_{ L^\infty(\Omega)} , \ \norm{b(t)}_{ L^\infty(\Omega)},\ \norm{c(t)}_{ L^\infty(\Omega)}\leq \mathcal H,  \text{ for all } t\ge 0.
 \end{equation}
\end{theorem}
\begin{remark}\label{remark:Boltzmann_entropy}
 In some senses, in \eqref{P2}, if we add two sides of the first reaction with $\mathcal{B}$, the second reaction with $\mathcal{C}$ and the third reaction with $\mathcal{A},$ then we obtain a reversible reaction system as follows.
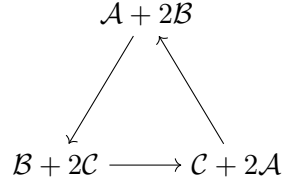
\begin{center}
    \begin{tikzpicture}[node distance=1.5cm]
    \node (A) at (2,2) {$\mathcal A + 2\mathcal{B}$};
    \node (B) at (0.8,0) {$\mathcal B + 2\mathcal{C}$};
    \node (C) at (3.2,0) {$\mathcal C + 2\mathcal{A}$};
    \draw[->] (A) -- (B);
    \draw[->] (B) -- (C);
    \draw[->] (C) -- (A);
\end{tikzpicture}
\end{center}
We can define the relative Boltzmann entropy corresponding to the positive equilibrium $(a_\infty, b_\infty, c_\infty)$
\[
\mathcal{E}(a,b,c| a^*_\infty,b^*_\infty,c^*_\infty)= \sum\limits_{z\in\{a,b,c\}}\int_\Omega\left(z\ln \dfrac{z}{z^*_\infty}-z+z^*_\infty\right )dx  ,  
\]
 and corresponding {\it entropy dissipation} along the flow of system \eqref{eq1'} is
\[
\begin{aligned}
\mathcal{D}(a,b,c)&=-\dfrac{d}{dt}\mathcal{E}(a,b,c| a^*_\infty,b^*_\infty,c^*_\infty)\\
&= d_1\int_{\Omega}\frac{|\nabla a|^2}{a}dx + d_2\int_{\Omega}\frac{|\nabla b|^2}{b}dx + d_3\int_{\Omega}\frac{|\nabla c|^2}{c}dx\\
&\; + \int_{\Omega}a(b-c)\ln\frac bc dx + \int_{\Omega}b(a-c)\ln \frac ac dx + \int_{\Omega}c(a-b)\ln \frac ab dx.
\end{aligned}
\]
However, due to the irreversibility of \eqref{P2} and the presence of the boundary equilibrium, we are not able prove an entropy-entropy dissipation inequality type
\[
\mathcal{D}(a,b,c) \geq H(t) \mathcal{E}(a,b,c| a^*_\infty,b^*_\infty,c^*_\infty),
\]
for some suitable $H(t)$, to achieve the convergence to the positive equilibrium point as in \cite{FT18}.
\end{remark}

Inheriting  techniques in previous sections, we now investigate the stability of  the positive equlibrium  $ 
(a^*_\infty,b^*_\infty,c^*_\infty)=(M/3,M/3,M/3) $ by the linearization method and the instability of the  boundary equilibria $ (\bar a_\infty,\bar b_\infty,\bar c_\infty) \in \{(0,0,M); (M,0,0); (0,M,0)\} $  by bootstrap instability.
	\subsection{Stability of positive equilibrium}
Without loss of generality, we assume that $M=3$, which gives the positive equilibrium $(a^*_\infty,b^*_\infty,c^*_\infty)=(1,1,1).$  We denote by $(c_1, c_2, c_3)=(a-1,b-1,c-1)$ and the system for $( c_1, c_2, c_3)$ can be written as
	\begin{equation}\label{eq2}
		\begin{cases}
			\partial_t c_1 -d_1\Delta c_1=-2c_1+ c_2 + c_3 -c_1 c_2 -c_1 c_3 +2 c_2  c_3 ,&\text { in } \Omega \times\mathbb R_+,\\
			\partial_t  c_2  -d_2\Delta  c_2 =c_1-2 c_2 + c_3 -c_1 c_2 +2c_1 c_3 - c_2  c_3 ,&\text { in } \Omega \times\mathbb R_+,\\
			\partial_t  c_3  -d_3\Delta  c_3 =c_1+ c_2 -2 c_3 +2c_1 c_2 -c_1 c_3 - c_2  c_3 ,&\text { in } \Omega \times\mathbb R_+,\\
			\nabla c_1 \cdot \nu= \nabla c_2 \cdot \nu=\nabla c_3 \cdot \nu =0,&\text { on } \partial\Omega \times\mathbb R_+,\\
			( c_1( \cdot,0), c_2(\cdot,0), c_3(\cdot,0)) =(a_0 -1,b_0 -1,c_0 -1),&\text { in } \Omega .
	\end{cases}\end{equation}
The mass conservation law for \eqref{eq1'} is 
\begin{equation*} 
    \int_\Omega (a(x,t)+b(x,t)+c(x,t) )dx = \int_\Omega (a_0(x)+b_0(x)+c_0(x) )dx =: M ,
\end{equation*}
which leads to \[
\int_\Omega (c_1(x,t)+c_2(x,t)+c_3(x,t) )dx=0.
\] 
To show the stability of  the positive equilibrium point for system \eqref{eq1'}, we use the similar arguments in the proof of Theorem \ref{thm5}:     showing that the linear part has a spectral gap; then nonlinear part, which only appears in quadratic form, is shown to be dominated by the linear one. 
	\begin{theorem}\label{thm1}
		Let the assumptions \eqref{A1} and \eqref{A2} hold. There is a $\epsilon>0$ \blue{depending on $M_1, M_2, \mathcal H$, $\Omega$, $d_1, d_2, d_3$} such that if
  \[
  \norm{a_0-a^*_\infty}^2_2+\norm{b_0-b^*_\infty}^2_2+\norm{c_0-c^*_\infty}^2_2 \leq \epsilon,
  \]
  then 
  \begin{equation*}
      \norm{a(t)-a^*_\infty}^2_2+\norm{b(t)-b^*_\infty}^2_2+\norm{c(t)-c^*_\infty}^2_2 \leq C e^{-\alpha_4 t} \text{ for all }t\geq 0
  \end{equation*}
  for some positive constants $C, \alpha_4$.
	\end{theorem}
	\begin{proof}
 Remember that the classical solution  to \eqref{eq1'} is  bounded uniformly in both time and space, so is the solution to \eqref{eq2}.
	Define the Lyapunov function	\[
 \mathcal L_2(t)  =  \left( \norm{c_1(t)}^2_2+\norm{ c_2 (t)}^2_2+\norm{ c_3 (t)}^2_2 \right). \]
		Along the solutions to \eqref{eq2}, we have
		\begin{equation}\label{eq6}
  \begin{aligned}
    & -\dfrac{1}{2}\dfrac{d\mathcal L_2(t)}{dt} = -\int_\Omega \left ( c_1 \partial_t c_1 +  c_2  \partial_t  c_2 + c_3  \partial_t  c_3 \right) dx\\
      &=-\sum_{i=1}^{3}d_i \int_\Omega c_i \Delta c_idx  
      -	 \int_\Omega \left( (-2c_1+c_2+c_3)c_1+ (c_1-2c_2+c_3)c_2+(c_1+c_2-2c_3)c_3 \right)dx \\
      &-\int_\Omega \left ( c_1(-c_1 c_2  -c_1 c_3 +2 c_2  c_3 )+ c_2 (-c_1 c_2 - c_2  c_3 +2c_1 c_3 )+ c_3 (2c_1 c_2 - c_2  c_3 -c_1 c_3 ) \right) dx.
  \end{aligned}
		\end{equation}
For the first term in the right-hand side of \eqref{eq6}, applying integration by parts and the Poincar\'e-Wirtinger inequality, we have
		\begin{equation} \label{eq4}
			-\sum_{i=1}^{3}d_i \int_\Omega c_i \Delta c_idx =\sum_{i=1}^{3}d_i  \norm{\nabla c_i}^2_2  \geq P(\Omega) \min \{d_1,d_2,d_3\} (\norm{c_1-\bar c_1 }^2_2+\norm{ c_2 -\bar c_2 }^2_2+\norm{ c_3 -\bar c_3 }^2_2).
		\end{equation}
	 Consider the second term in the right-hand side of \eqref{eq6}, one gets
		\begin{align*}
			&-\int_\Omega \left( (-2c_1+c_2+c_3)c_1+ (c_1-2c_2+c_3)c_2+(c_1+c_2-2c_3)c_3 \right)dx\\
   &=2\int_\Omega(c_1^2+ c_2 ^2+ c_3 ^2-c_1 c_2 -c_1 c_3 - c_2  c_3 ) dx\\
			&=\dfrac{1}{3} \int_\Omega \left( (c_1+ c_2 -2 c_3 )^2+(c_1-2 c_2 + c_3 )^2+(-2c_1+ c_2 + c_3 )^2 \right) dx\\
			&\geq 
			\dfrac{1}{3  }  \left(\int_\Omega (c_1+ c_2 -2 c_3 )dx \right)^2+\dfrac{1}{3  }\left(\int_\Omega (c_1-2 c_2 + c_3 )dx \right)^2+\dfrac{1}{3  }\left(\int_\Omega (-2c_1+ c_2 + c_3 )dx \right)^2\\   
			&=     \left( 2\bar c_1^2+ 2\bar c_2^2+2\bar c_3^2 - 2\bar c_1\bar c_2-2\bar c_1\bar c_3-2\bar c_2\bar c_3 \right) .
		\end{align*}
		This, together with the conservation law $\bar c_1+\bar c_2+\bar c_3=0$, yields
		\begin{equation}\label{eq5}
			-\int_\Omega \left( (-2c_1+c_2+c_3)c_1+ (c_1-2c_2+c_3)c_2+(c_1+c_2-2c_3)c_3 \right)dx \geq  3 \left( \norm{ \bar c_1 }^2_2+\norm{ \bar c_2 }^2_2+\norm{ \bar c_3 }^2_2  \right).
		\end{equation}
 Combining \eqref{eq6}, \eqref{eq4} and \eqref{eq5}, one obtains
 \begin{align*}
  \dfrac{d\mathcal L_2(t)}{dt} &\leq - 2\alpha_4 \mathcal{L}_2(t) \\
 &+2\int_\Omega \left ( c_1(-c_1 c_2  -c_1 c_3 +2 c_2  c_3 )+ c_2 (-c_1 c_2 - c_2  c_3 +2c_1 c_3 )+ c_3 (2c_1 c_2 - c_2  c_3 -c_1 c_3 ) \right)dx,
 \end{align*}
 where $\alpha_4=\min \{ 3,P(\Omega)d_1,P(\Omega)d_2,P(\Omega)d_3\}.$
  The remainder of the proof follows the same as in the proof of Theorem \ref{thm5}, so we omit it here.
 \end{proof}
	\subsection{Instability of boundary equilibrium} Due to the symmetry of the problem \eqref{eq1'}, without loss of generality, we only consider the instability of the boundary equilibrium $(\bar a_\infty,0,0).$ Denote by new variables as perturbation around the boundary equilibria $v=(a-\bar a_\infty,b,c).$ Then, system \eqref{eq1'} can be rewritten in $v$ as follows
	\begin{equation}\label{eq8}
		\begin{cases}
			\partial_t v_1 -d_1\Delta v_1=-\bar a_\infty v_2-\bar a_\infty v_3-v_1v_2-v_1v_3+2v_2v_3,&\text { in } \Omega \times\mathbb R_+,\\
			\partial_t v_2 -d_2\Delta v_2=-\bar a_\infty v_2+2\bar a_\infty v_3-v_1v_2+2v_1v_3-v_2v_3,&\text { in } \Omega \times\mathbb R_+,\\
			\partial_t v_3-d_3\Delta v_3=2 \bar a_\infty v_2-\bar a_\infty v_3+2v_1v_2-v_1v_3-v_2v_3,&\text { in } \Omega \times\mathbb R_+,\\
			\nabla v_i \cdot \nu =0,\ 1\leq i\leq 3,&\text { on } \partial \Omega \times\mathbb R_+,\\
			v(x,0)=v_{0}(x)=(a_0-\bar a_\infty,b_0,c_0),&\text { in } \Omega.
	\end{cases}\end{equation}
	or equivalently $v_t=L_1v+N_1(v)$
	\begin{equation*}
        L_1=  \begin{pmatrix} 
		d_1\Delta & -\bar a_\infty & -\bar a_\infty \\
		0 & d_2\Delta -\bar a_\infty & 2\bar a_\infty\\
		0 & 2 \bar a_\infty & d_3\Delta -\bar a_\infty \\
	\end{pmatrix}
        \quad \text{ and } \quad 
        N_1(v)=\begin{pmatrix} 
		-v_1v_2-v_1v_3+2v_2v_3  \\
		-v_1v_2+2v_1v_3-v_2v_3\\
		2v_1v_2-v_1v_3-v_2v_3\\
	\end{pmatrix}.
        \end{equation*}
Taking the derivative with respect to time, we get
 \begin{equation*}  v_{tt}=L_1v_t+\partial_t\left (N_1(v)\right ).
 \end{equation*}
 Denote by a new variable  $y=(v,v_t)^T,$ then $y$ satisfies the following system 
 \begin{equation}\label{eq16}
     y_t=Ly+N(y),
 \end{equation}
 where  $L=  \begin{pmatrix} 
		L_1 & 0 \\
		0 & L_1\\
	\end{pmatrix}$
	and $N(y)=\begin{pmatrix} 
		N_1(v)  \\
		\partial_t\left (N_1(v)\right )\\
	\end{pmatrix}.$
	As before, we consider the following norms  \[
\norm{y}=\norm{v}_2+\norm{v_t}_2;\  \vertiii{y}   = \norm {v}_{H^2(\Omega)}+\norm{v_t}_{2} .
    \] 
We are now able to present the main theorem of this section.
	\begin{theorem}\label{thm3} Let $\Omega \subset\mathbb R^n,$ for $n\in \{1,2,3\}$,   and assume that  \eqref{A1}-\eqref{A2} hold. \blue{Then there exist positive constants $\theta_0, \tau_0$ depending only on $M, d_1, d_2, d_3, \Omega, \mathcal{H}$ and $ \int_{\Omega}(b_0+c_0)(x)dx$ such that if $y(0) = \delta y_0$ with $\|y_0\| = 1$ and $\delta$ sufficiently small, then
    \begin{equation*}
        \|y(t)- \delta e^{Lt}y_0\| \le C\delta^2e^{2\bar a_\infty t},
    \end{equation*}
    for all
    \begin{equation*}
        0\leq t\leq T^\delta:=\dfrac{1}{\bar a_\infty}\ln \dfrac{\theta_0}{\delta},
    \end{equation*}
    and, at the escape time $T^\delta$, it holds
    \begin{equation*}
        \|y(T^\delta)\| \ge \tau_0 > 0.
    \end{equation*}}
\end{theorem}
    
To prove Theorem \ref{thm3}, we can use similar arguments as in the proof of Theorem \ref{thm4}. Since \eqref{eq1'} is similar to \eqref{p1} in the sense that the right hand side are purely quadratic and the solutions is uniformly bounded in time, similar results to Lemmas \ref{lm4} and \ref{lm5} also hold true for \eqref{eq1'}. 
We only need to check the bounds of the semigroup $\{e^{Lt}\}$, which are shown in the following lemmas.
\begin{lemma}\label{lm3}
	 There exists a positive constant $C_L$ such that $$\norm{ e^{Lt}}_{(L^2(\Omega),L^2(\Omega))} \leq C_Le^{ \bar a_\infty t}  .$$
	\end{lemma}
	\begin{proof}
	    Recall that $0=\lambda_0>\lambda_1\geq \lambda_2\geq \dots\geq \lambda_n\geq\dots \to -\infty$ are the eigenvalues of the Neumann Laplacian. Then, for  an index $i$ fixed, eigenvalues $\eta$ of the linear operator $L$ are described by
\begin{equation*}
       \begin{vmatrix}
          d_1\lambda_i-\eta  & -\bar a_\infty & -\bar a_\infty \\
           0 & d_2\lambda_i -\bar a_\infty-\eta& 2\bar a_\infty\\
           0 & 2 \bar a_\infty& d_3\lambda_i -\bar a_\infty-\eta
       \end{vmatrix}
       =0.
\end{equation*}
That means, $\eta=d_1\lambda_i\leq 0,$ or 
\begin{equation*} 
    \eta^2 -\left( (d_2+d_3)\lambda_i-2 \bar a_\infty\right) \eta + d_2d_3\lambda_i^2-(d_2+d_3)\lambda_i \bar a_\infty-3\bar a_\infty^2=0.
\end{equation*}
This equation always has two distinct roots, the larger one is determined as follows
\[
\eta_{\max}=\dfrac{(d_2+d_3)\lambda_i-2 \bar a_\infty+\sqrt{(d_2-d_3)^2\lambda_i^2 +16 \bar a_\infty^2}}{2}:=g(\lambda_i).
\]
By simple calculation, we have
\begin{align*}
 g'(\lambda_i)&=\dfrac{1}{2}\left(  d_2+d_3+\dfrac{(d_2-d_3)^2\lambda_i}{\sqrt{(d_2-d_3)^2\lambda_i^2+16\bar a_\infty^2}}  \right)\\
 &=\dfrac{2d_2d_3(d_2-d_3)^2\lambda_i^2+8(d_2+d_3)^2\bar a_\infty^2}{\sqrt{(d_2-d_3)^2\lambda_i^2+16\bar a_\infty^2}\left ( (d_2+d_3)\sqrt{(d_2-d_3)^2\lambda_i^2+16\bar a_\infty^2}-(d_2-d_3)^2\lambda_i \right)}\geq 0.
\end{align*}
 Therefore,  $g$    is   increasing, thus,    $\eta_{\max} \leq g(0)=\bar a_\infty.$ That means, all eigenvalues of $L$ is bounded above by $\bar a_\infty$. Lemma \ref{lm3} is proved.
	\end{proof}

	\begin{lemma}
		For any $\norm{y_0}=1$ with $C_1=\dfrac{1}{2}\int_\Omega (y_{0,1}(x)+y_{0,2}(x)) dx \ne 0,$ there exists a positive constant $C_P$ \blue{depending only on $\Omega$ and $y_0$} such that \[
  \norm {e^{Lt}y_0}  \geq C_P e^{\bar a_\infty t}, \] for all $t>0.$
	\end{lemma}
	\begin{proof}
		Let us consider the linear system
		\begin{equation}\label{eq9}
			\begin{cases}
				\partial_t \wh v_1 -d_1\Delta \wh v_1=-\bar a_\infty \wh v_2-\bar a_\infty \wh v_3,&\text { in } \Omega \times\mathbb R_+,\\
				\partial_t \wh v_2 -d_2\Delta \wh v_2=-\bar a_\infty \wh v_2+2\bar a_\infty \wh v_3,&\text { in } \Omega \times\mathbb R_+,\\
				\partial_t v_3-d_3\Delta \wh v_3=2\bar a_\infty \wh v_2-\bar a_\infty \wh v_3,&\text { in } \Omega \times\mathbb R_+,\\
				\nabla \wh v_i \cdot \nu =0,\ 1\leq i\leq 3, &\text { on } \partial\Omega \times\mathbb R_+,\\
				\wh v(x,0)=y_0(x),&\text { in } \Omega .
		\end{cases}\end{equation}   
Taking the integration over $\Omega,$ one obtains form \eqref{eq9} that
		\begin{equation}\label{l1}
			\begin{cases}
				\frac{d}{dt} \int_\Omega \wh v_1(x,t) dx =-\bar a_\infty \int_\Omega \wh v_2(x,t) dx-\bar a_\infty \int_\Omega \wh v_3(x,t) dx,\\
				\frac{d}{dt} \int_\Omega \wh v_2(x,t) dx=-\bar a_\infty\int_\Omega \wh v_2(x,t) dx+2\bar a_\infty\int_\Omega \wh v_3(x,t) dx,\\
				\frac{d}{dt} \int_\Omega \wh v_3(x,t) dx=2\bar a_\infty\int_\Omega \wh v_2(x,t) dx-\bar a_\infty \int_\Omega \wh v_3(x,t) dx,\\
		\end{cases}\end{equation}  
		Therefore, we can compute
		\begin{equation}\label{l2_1}
			\begin{cases}
				\int_\Omega \wh v_1(x,t) dx = -2C_1 e^{\bar a_\infty t}+C_3,\\
				\int_\Omega \wh v_2(x,t) dx= C_1e^{\bar a_\infty t}+C_2e^{-3\bar a_\infty t},\\
				\int_\Omega \wh v_3(x,t) dx=C_1e^{\bar a_\infty t}-C_2 e^{-3\bar a_\infty t} ,\\
		\end{cases}\end{equation} 
		where the constants $C_1=\frac{1}{2}\int_\Omega (y_{0,2}(x)+y_{0,3}(x)) dx,\ C_2=\frac{1}{2}\int_\Omega (y_{0,2}(x)-y_{0,3}(x)) dx$ and $C_3=\int_\Omega (y_{0,1}(x) + y_{0,2}(x) + y_{0,3}(x)) dx.$
		Thus,
		\begin{align*}
			\norm{e^{Lt}y_0}_2^2&= \int_\Omega \left( |\wh v_1(x,t)|^2+|\wh v_2(x,t)|^2+|\wh v_3(x,t)|^2\right) dx\notag \\ 
			& \geq  \left ( \left (\int_\Omega \wh v_1(x,t)dx \right)^2+\left (\int_\Omega \wh v_2(x,t)dx \right)^2+\left(\int_\Omega \wh v_3(x,t)dx \right)^2 \right) \notag \\
			& =  \left( 6C_1^2 e^{2\bar a_\infty t} -2C_1C_3 e^{\bar a_\infty t}+C_3^2+2C_2^2e^{-6\bar a_\infty t} \right)\notag  \\
   & =  \left( 5C_1^2 e^{2\bar a_\infty t} + (C_1e^{\bar a_\infty t}-C_3)^2+2C_2^2e^{-6\bar a_\infty t} \right)\notag  \\
			&\geq 5C_1^2 e^{2\bar a_\infty t}  .
		\end{align*}
		  Hence,
		\begin{equation}\label{l3}
		    \norm{e^{Lt}y_0}_2\geq \sqrt{5}|C_1| e^{\bar a_\infty t}.
		\end{equation}  
From \eqref{l2_1}, we obtain
\begin{equation*} 
			\begin{cases}
				\frac{d}{dt} \int_\Omega \wh v_1(x,t) dx =2C_2e^{-3\bar a_\infty t},  \\
				\frac{d}{dt} \int_\Omega \wh v_2(x,t) dx=C_1e^{\bar a_\infty t}-5C_2e^{-3\bar a_\infty t}, \\
				\frac{d}{dt} \int_\Omega \wh v_3(x,t) dx=C_1e^{\bar a_\infty t}+3C_2e^{-3\bar a_\infty t}. \\
		\end{cases}\end{equation*}  
  Then,
 \begin{align*}
      \norm{\partial_t \wh v  }_2^2&=\int_\Omega\left ((\partial_t \wh v_1)^2+ (\partial_t \wh v_2)^2+(\partial_t \wh v_3)^2 \right) dx\\
      &\ge \frac 13\int_{\Omega}\bra{\partial_t \wh v_1 + \partial_t \wh v_2 + \partial_t \wh v_3}^2dx\\
      &\geq  \frac 13\left (\int_\Omega (\partial_t \wh v_1 +  \partial_t \wh v_2 + \partial_t \wh v_3 )  dx\right)^2\\
      &=\frac 13 \left(2C_1^2e^{2\bar a_\infty t}+ 38C_2^2e^{-6\bar a_\infty t}-4C_1C_2e^{-2\bar a_\infty t}  \right)\\
      &\geq \frac 13C_1^2 e^{2\bar a_\infty t} .
 \end{align*}
By this,
\begin{equation}\label{l4}
\norm{\partial_t(e^{Lt}y_0)}_2\geq   \frac 13|C_1|  e^{\bar a_\infty t}.
\end{equation}
It can be deduced from \eqref{l3} and \eqref{l4} that
  \[
  \norm {e^{Lt}y_0}  \geq C_P e^{\bar a_\infty t},
  \]
  where the constant $C_P= (\sqrt{5}+1/3)|C_1|   $ depends on $C_1$ and $\Omega.$
	\end{proof}
	
\subsection*{Acknowledgement}  
The authors would like to thank Prof. Laurent Desvillettes and Prof. Klemens Fellner for their fruitful discussions.

The first author was supported by the Ernst Mach scholarship from OeAD, Austria's Agency for Education and Internationalization, Mobility Programmers and Cooperation. The work was completed during the first author's visit to University of Graz, and the university's hospitality is greatly acknowledged.

\providecommand{\bysame}{\leavevmode\hbox to3em{\hrulefill}\thinspace}
\providecommand{\MR}{\relax\ifhmode\unskip\space\fi MR }
\providecommand{\MRhref}[2]{%
	\href{http://www.ams.org/mathscinet-getitem?mr=#1}{#2}
}
\providecommand{\href}[2]{#2}

\end{document}